\documentclass{amsart}
\usepackage{amsmath,amsthm,amsfonts,amssymb}

\numberwithin{equation}{section}
\theoremstyle{plain}
\newtheorem{corollary}{Corollary}
\newtheorem{definition}{Definition}
\newtheorem{lemma}{Lemma}
\newtheorem{proposition}{Proposition}
\newtheorem{remark}{Remark}
\newtheorem{theorem}{Theorem}

\newcommand{\one}{{{\rm 1\mkern-1.5mu}\!{\rm I}}}

\begin{document}

\title[Quenched Large Deviations for RWRE]{Quenched Large Deviations for \\Random Walk in a Random Environment}
\author{Atilla Yilmaz}
\address
{Department of Mathematics\newline
\indent Weizmann Institute of Science\newline
\indent Rehovot 76100\newline
\indent ISRAEL}%
\email{atilla.yilmaz@weizmann.ac.il}%
\urladdr{http://www.wisdom.weizmann.ac.il/$\sim$yilmaz/}
\date{March 27, 2008. Revised on December 15, 2008.}

\thanks{This research was supported partially by a grant from the National Science Foundation: DMS-06-04380.}

\maketitle

\begin{abstract}
We take the point of view of a particle performing random walk with bounded jumps on $\mathbb{Z}^d$ in a stationary and ergodic random environment. We prove the quenched large deviation principle (LDP) for the pair empirical measure of the environment Markov chain. By an appropriate contraction, we deduce the quenched LDP for the mean velocity of the particle and obtain a variational formula for the corresponding rate function. We propose an Ansatz for the minimizer of this formula. When $d=1$, we verify this Ansatz and generalize the nearest-neighbor result of Comets, Gantert and Zeitouni to walks with bounded jumps.
\end{abstract}

\section{Introduction}

\subsection{The model}

The random motion of a particle on $\mathbb{Z}^d$ can be modelled by a discrete time Markov chain. Write $\pi(x,x+z)$ for the transition probability from $x$ to $x+z$ for each $x,z\in\mathbb{Z}^d$ and refer to $\omega_x:=(\pi(x,x+z))_{z\in\mathbb{Z}^d}$ as the environment at $x$. If the environment $\omega:=(\omega_x)_{x\in\mathbb{Z}^d}$ is sampled from a probability space $(\Omega,\mathcal{B},\mathbb{P})$, then the particle is said to perform random walk in a random environment (RWRE). Here, $\mathcal{B}$ is the Borel $\sigma$-algebra.

For each $z\in\mathbb{Z}^d$, define the shift $T_z$ on $\Omega$ by $\left(T_z\omega\right)_x=\omega_{x+z}$ and assume that $\mathbb{P}$ is stationary and ergodic under $\left(T_z\right)_{z\in\mathbb{Z}^d}$. Plus, assume that the jumps are bounded by a constant $B$, i.e., for any $z=(z_1,\ldots,z_d)\in\mathbb{Z}^d$, $\pi(0,z)=0$ $\mathbb{P}$-a.s.\ unless $0<|z_1|+\cdots+|z_d|\leq B$. Denote the set of allowed jumps of the walk by \[\mathcal{R} := \{(z_1,\ldots,z_d)\in\mathbb{Z}^d:\;0<|z_1|+\cdots+|z_d|\leq B\}.\] When $B=1$, the walk is said to be nearest-neighbor and the set of allowed jumps is \[U:=\{(z_1,\ldots,z_d)\in\mathbb{Z}^d:\;|z_1|+\cdots+|z_d|=1\}.\]

For any $x\in\mathbb{Z}^d$ and $\omega\in\Omega$, the Markov chain with transition probabilities given by $\omega$ induces what is called the ``quenched" probability measure $P_x^\omega$ on the space of paths starting at $x$. The semi-direct product $P_x:=\mathbb{P}\times P_x^\omega$ is referred to as the ``averaged" measure. Expectations under $\mathbb{P}$, $P_x^\omega$ and $P_x$ are denoted by $\mathbb{E}$, $E_x^\omega$ and $E_x$, respectively.

Because of the extra layer of randomness in the model, the standard questions of recurrence vs.\ transience, the law of large numbers (LLN), the central limit theorem (CLT) and the large deviation principle (LDP) --- which have well known answers for classical random walk --- become hard. However, it is possible by taking the ``point of view of the particle" to treat the two layers of randomness as one: If we denote the random path of the particle by $X:=(X_n)_{n\geq0}$, then $(T_{X_n}\omega)_{n\geq0}$ is a Markov chain (referred to as the ``environment Markov chain") on $\Omega$ with transition kernel $\overline{\pi}$ given by \[\overline{\pi}(\omega,\omega'):=\sum_{z:T_z\omega=\omega'}\pi(0,z).\] This is a standard approach in the study of random media. (See, for example, \cite{DeMasi}, \cite{KV}, \cite{Kozlov}, \cite{Olla} or \cite{PV}.)

Instead of viewing the environment Markov chain as an auxiliary construction, one can introduce it first and then deduce the particle dynamics from it.
\begin{definition}\label{ortamkeli}
A function $\hat{\pi}:\Omega\times\mathcal{R}\to\mathbb{R}^+$ is said to be an ``environment kernel" if\\(i) $\hat{\pi}(\cdot,z)$ is $\mathcal{B}$-measurable for each $z\in\mathcal{R}$, and (ii) $\sum_{z\in\mathcal{R}}\hat{\pi}(\cdot,z)=1,\ \mathbb{P}$-a.s.\\
It can be viewed as a transition kernel on $\Omega$ via the following identification: \[\overline{\pi}(\omega,\omega'):=\sum_{z:T_z\omega=\omega'}\hat{\pi}(\omega,z).\]

Given $x\in\mathbb{Z}^d$, $\omega\in\Omega$ and any environment kernel $\hat{\pi}$, the ``quenched" probability measure $P_x^{\hat{\pi},\omega}$ on the space of particle paths $(X_n)_{n\geq0}$ starting at $x$ in environment $\omega$ is defined by setting $P_x^{\hat{\pi},\omega}\left(X_o=x\right)=1$ and \[P_x^{\hat{\pi},\omega}\left(X_{n+1}=y+z\left|X_n=y\right.\right)=\hat{\pi}(T_y\omega,z)\] for all $n\geq0$, $y\in\mathbb{Z}^d$ and $z\in\mathcal{R}$. The semi-direct product $P_x^{\hat{\pi}}:=\mathbb{P}\times P_x^{\hat{\pi},\omega}$ is referred to as the ``averaged" measure, and expectations under $P_x^{\hat{\pi},\omega}$ and $P_x^{\hat{\pi}}$ are denoted by $E_x^{\hat{\pi},\omega}$ and $E_x^{\hat{\pi}}$, respectively.
\end{definition}

See \cite{Sznitman} or \cite{rwrereview} for a more detailed description of RWRE, examples and a survey of the literature. In this work, we focus on the quenched large deviation properties of this model.

\subsection{Previous results}\label{egridogru}

Greven and den Hollander \cite{GdH} prove the quenched LDP for the mean velocity of a particle performing nearest-neighbor random walk on $\mathbb{Z}$ in a product environment (i.e., when $\mathbb{P}$ is a product measure) and show that the rate function is convex but typically has parts consisting of line segments. Their proof makes use of an auxiliary branching process formed by the excursions of the walk. Using a completely different technique, Comets, Gantert and Zeitouni \cite{CGZ} extend the results of \cite{GdH} to stationary and ergodic environments. Their argument involves first proving a quenched LDP for the passage times of the walk by an application of the G\"artner--Ellis theorem and then inverting this to get the desired LDP for the mean velocity.

For $d\geq2$, the first result on quenched large deviations is given by Zerner \cite{Zerner}. He uses a subadditivity argument again for certain passage times to get the quenched LDP in the case of product environments. He assumes that the environment is ``nestling", i.e., the convex hull of the support of the law of $\sum_{z\in\mathcal{R}}\pi(0,z)z$ contains the origin. By a more direct use of the subadditive ergodic theorem, Varadhan \cite{Raghu} drops the nestling assumption and generalizes Zerner's result to stationary and ergodic environments.

The drawback of using subadditivity arguments is that one does not obtain a formula for the rate function. Rosenbluth \cite{jeffrey} takes the point of view of the particle and gives an alternative proof of the quenched LDP for the mean velocity in the case of stationary and ergodic environments. He provides a variational formula for the rate function. His approach is parallel to the work of Kosygina, Rezakhanlou and Varadhan \cite{KRV} on diffusions in random environments.

\subsection{Our results}\label{results}

For any measurable space $(Y,\mathcal{F})$, write $M_1(Y,\mathcal{F})$ (or simply $M_1(Y)$ whenever no confusion occurs) for the space of probability measures on $(Y,\mathcal{F})$. Consider random walk $X=(X_n)_{n\geq0}$ on $\mathbb{Z}^d$ in a stationary and ergodic random environment, and focus on \[\nu_{n,X} := \frac{1}{n}\sum_{k=0}^{n-1}\one_{T_{X_k}\omega,X_{k+1}-X_k}\] which is a random element of $M_1(\Omega\times\mathcal{R})$. The map $(\omega,z)\mapsto(\omega,T_z\omega)$ imbeds $M_1(\Omega\times\mathcal{R})$ into $M_1(\Omega\times\Omega)$, and we therefore refer to $\nu_{n,X}$ as the pair empirical measure of the environment Markov chain. For any $\mu\in M_1(\Omega\times\mathcal{R})$, define the probability measures $(\mu)^1$ and $(\mu)^2$ on $\Omega$ by \[\mathrm{d}(\mu)^1(\omega):=\sum_{z\in\mathcal{R}}\mathrm{d}\mu(\omega,z)\ \ \mbox{and}\ \ \mathrm{d}(\mu)^2(\omega):=\sum_{z\in\mathcal{R}}\mathrm{d}\mu(T_{-z}\omega,z)\] which are the marginals of $\mu$ when $\mu$ is seen as an element of $M_1(\Omega\times\Omega)$. With this notation, set
\begin{equation*}
M_{1,s}^{\ll}(\Omega\times\mathcal{R}):= \left\{\mu\in M_1(\Omega\times\mathcal{R}): (\mu)^1=(\mu)^2\ll\mathbb{P},\ \frac{\mathrm{d}\mu(\omega,z)}{\mathrm{d}(\mu)^1(\omega)}>0\ \mathbb{P}\mbox{-a.s.\ for each }z\in U\right\}.
\end{equation*}

Our first result is the following theorem whose proof constitutes Section \ref{pairLDPsection}.
\begin{theorem}\label{level2LDP} If there exists an $\alpha>0$ such that
\begin{equation}
\int|\log \pi(0,z)|^{d+\alpha}\,\mathrm{d}\mathbb{P}<\infty\label{kimimvarki}
\end{equation} for each $z\in\mathcal{R}$, then
$\mathbb{P}$-a.s.\ $(P_o^\omega(\nu_{n,X}\in\cdot))_{n\geq1}$ satisfy the LDP with the good rate function $\mathfrak{I}^{**}$, the double Fenchel--Legendre transform of $\mathfrak{I}:M_1(\Omega\times\mathcal{R})\to\mathbb{R}^+$ given by
\begin{equation}
\mathfrak{I}(\mu)=\left\{
\begin{array}{ll}
\int_{\Omega}\sum_{z\in\mathcal{R}} \mathrm{d}\mu(\omega,z)\log\frac{\mathrm{d}\mu(\omega,z)}{\mathrm{d}(\mu)^1(\omega)\pi(0,z)}& \mbox{if }\mu\in M_{1,s}^{\ll}(\Omega\times\mathcal{R}),\\ \infty & \mbox{otherwise.}
\end{array}\right.\label{level2ratetilde}
\end{equation}
\end{theorem}
\begin{remark}
$\mathfrak{I}$ is convex, but it may not be lower semicontinuous. Therefore, $\mathfrak{I}^{**}$ is not a-priori equal to $\mathfrak{I}$. See Appendix A for a detailed explanation.
\end{remark}

We start Section \ref{birboyutadonus} by deducing the quenched LDP for the mean velocity of the particle by an application of the contraction principle. For any $\mu\in M_1(\Omega\times\mathcal{R})$, set
\begin{equation}
\xi_{\mu}:=\int\sum_{z\in\mathcal{R}}\mathrm{d}\mu(\omega,z)z.\label{ximu}
\end{equation} For any $\xi\in\mathbb{R}^d$, define
\begin{equation}
A_\xi:=\{\mu\in M_1(\Omega\times\mathcal{R}):\xi_{\mu}=\xi\}.\label{Axi}
\end{equation} The corollary below follows immediately from Theorem \ref{level2LDP} and reproduces the central result of \cite{jeffrey}.

\begin{corollary}\label{level1LDP}
If there exists $\alpha>0$ such that (\ref{kimimvarki}) holds for each $z\in\mathcal{R}$, then $(P_o^\omega(\frac{X_n}{n}\in\cdot))_{n\geq1}$ satisfy the LDP for $\mathbb{P}$-a.e.\ $\omega$. The good rate function $I$ is given by
\begin{eqnarray}
I(\xi)&=&\inf_{\mu\in A_\xi} \mathfrak{I}^{**}(\mu)\label{level1rate}\\
&=&\inf_{\mu\in A_\xi} \mathfrak{I}(\mu)\label{level1ratetilde}
\end{eqnarray} where $\mathfrak{I}$ and $A_\xi$ are defined in (\ref{level2ratetilde}) and (\ref{Axi}), respectively. $I$ is convex.
\end{corollary}

One would like to get a more explicit expression for the rate function $I$. This is not an easy task in general. $M_1(\Omega\times\mathcal{R})$ is compact (when equipped with the weak topology), $A_\xi$ is closed, and $\mathfrak{I}^{**}$ is lower semicontinuous. Therefore, the infimum in (\ref{level1rate}) is attained. However, due to the possible lack of lower semicontinuity of $\mathfrak{I}$, the infimum in (\ref{level1ratetilde}) may not be attained. 
\begin{definition} \label{K}
A measurable function $F:\Omega\times\mathcal{R}\rightarrow\mathbb{R}$ is said to be in class $\mathcal{K}$ if it satisfies the following conditions:
\begin{description}
\item[Moment] For each $z\in\mathcal{R}$, $F(\cdot,z)\in\bigcup_{\alpha>0}L^{d+\alpha}(\mathbb{P})$.
\item[Mean zero] For each $z\in\mathcal{R}$, $\mathbb{E}\left[F(\cdot,z)\right]=0$.
\item[Closed loop] For $\mathbb{P}$-a.e.\ $\omega$ and any finite sequence $(x_{k})_{k=0}^n$ in $\mathbb{Z}^d$ such that $x_{k+1}-x_k\in\mathcal{R}$ and $x_0=x_n$, \[\sum_{k=0}^{n-1}F(T_{x_k}\omega,x_{k+1}-x_k)=0.\]
\end{description}
\end{definition}
\noindent The following lemma provides an Ansatz and states that whenever an element of $A_\xi$ fits this Ansatz, it is the unique minimizer of (\ref{level1ratetilde}). Its proof concludes Section \ref{birboyutadonus}.

\begin{lemma}\label{lagrange}
For any $\xi\in\mathbb{R}^d$, if there exists a $\mu_\xi\in A_\xi\cap M_{1,s}^{\ll}(\Omega\times\mathcal{R})$ such that \[\mathrm{d}\mu_\xi(\omega,z)=\mathrm{d}(\mu_\xi)^1(\omega) \pi(0,z)\mathrm{e}^{\langle\theta,z\rangle+F(\omega,z)+ r}\] for some $\theta\in\mathbb{R}^d$, $F\in\mathcal{K}$ and $r\in\mathbb{R}$, then $\mu_\xi$ is the unique minimizer of (\ref{level1ratetilde}).
\end{lemma}

In Sections \ref{vandiseksin} and \ref{dolapdere}, we show that the recipe given in Lemma \ref{lagrange} works when $d=1$. We make the following assumptions:
\begin{enumerate}
\item [(A1)] There exists an $\alpha>0$ such that $\int|\log\pi(0,z)|^{1+\alpha}\mathrm{d}\mathbb{P}<\infty$ for each $z\in\mathcal{R}$.
\item [(A2)] There exists a $\delta>0$ such that $\mathbb{P}(\pi(0,\pm 1)\geq\delta)=1$. This is called ``uniform ellipticity".
\end{enumerate}
For every $y\in\mathbb{Z}$,
\begin{equation}
\tau_y:=\inf\{k\geq0:X_k\geq y\}\quad\mbox{and}\quad\bar{\tau}_y:=\inf\{k\geq0:X_k\leq y\}\label{nihatgomermis}
\end{equation} denote the right and left passage times of the walk. The following lemma is central to our argument.

\begin{lemma}\label{lifeisrandom}
Suppose $d=1$. Under the assumptions (A1) and (A2), the limits
$$\lambda(r):=\lim_{n\to\infty}\frac{1}{n}\log E_o^\omega\left[\mathrm{e}^{r\tau_n},\tau_n<\infty\right]\quad\mbox{and}\quad\bar{\lambda}(r):=\lim_{n\to\infty}\frac{1}{n}\log E_o^\omega\left[\mathrm{e}^{r\bar{\tau}_{-n}},\bar{\tau}_{-n}<\infty\right]$$ exist for $\mathbb{P}$-a.e. $\omega$. The functions $r\mapsto\lambda(r)$ and $r\mapsto\bar{\lambda}(r)$ are
\begin{enumerate}
\item [(i)] deterministic,
\item [(ii)] finite precisely on $(-\infty,r_c]$ for some $r_c\in[0,\infty)$, and
\item [(iii)] strictly convex and differentiable on $(-\infty,r_c)$.
\end{enumerate}
The constants $$\xi_c:=\left(\lambda'(r_c-)\right)^{-1}\quad\mbox{and}\quad\bar{\xi}_c:=-\left(\bar{\lambda}'(r_c-)\right)^{-1}$$
satisfy $-B<\bar{\xi}_c\leq0\leq\xi_c<B$. 
\end{lemma}

For every $\xi\in(-B,\bar{\xi}_c)\cup(\xi_c,B)$, we construct a $\mu_\xi$ that fits the Ansatz given in Lemma \ref{lagrange}. Substituting it in (\ref{level2ratetilde}), we get an explicit expression for (\ref{level1ratetilde}).

\begin{theorem}\label{explicitformulah}
Suppose $d=1$. Under the assumptions (A1) and (A2),
\begin{equation}\label{nnayni}
I(\xi)=\left\{
\begin{array}{ll}
\sup_{r\in\mathbb{R}}\{r-\xi\lambda(r)\}&\mbox{if }\xi>0,\\
\sup_{r\in\mathbb{R}}\{r+\xi\bar{\lambda}(r)\}&\mbox{if }\xi<0,\\
r_c&\mbox{if }\xi=0
\end{array}\right.
\end{equation}
with $r_c$ given in Lemma \ref{lifeisrandom}. The function $\xi\mapsto I(\xi)$ is
\begin{enumerate}
\item [(i)] affine linear on $[\bar{\xi}_c,0]$ and $[0,\xi_c]$,
\item [(ii)] strictly convex on $(-B,\bar{\xi}_c)$ and $(\xi_c,B)$, and
\item [(iii)] differentiable on $(-B,0)$ and $(0,B)$.
\end{enumerate}
\end{theorem}
Section \ref{vandiseksin} focuses on nearest-neighbor walks under assumption (A1). In that case, the proof of Lemma \ref{lifeisrandom} is straightforward since
$$\lambda(r)=\mathbb{E}\left(\log E_o^\omega\left[\mathrm{e}^{r\tau_1},\tau_1<\infty\right]\right)\quad\mbox{and}\quad\bar{\lambda}(r)=\mathbb{E}\left(\log E_o^\omega\left[\mathrm{e}^{r\bar{\tau}_{-1}},\bar{\tau}_{-1}<\infty\right]\right).$$ Naturally, Theorem \ref{explicitformulah} is identical to the quenched LDP result of \cite{CGZ}.

The general case of walks with bounded jumps is studied in Section \ref{dolapdere} where the proofs are more technical. Theorem \ref{explicitformulah} generalizes the quenched LDP result of \cite{CGZ}, but there is a qualitative difference:
\begin{proposition}\label{Wronskian}
Suppose $d=1$. For nearest-neighbor walks, $I(\xi)=I(-\xi)+\xi\cdot(\bar{\lambda}(0)-\lambda(0))$ if $\xi\in[-1,0)$.
Such a symmetry is generally absent for walks with bounded jumps.
\end{proposition}

\section{Large deviation principle for the pair empirical measure}\label{pairLDPsection}

As mentioned in Subsection \ref{egridogru}, Rosenbluth \cite{jeffrey} takes the point of view of a particle performing RWRE and proves the quenched LDP for the mean velocity. In this section, we generalize his argument and prove Theorem \ref{level2LDP}. The strategy is to show the existence of the logarithmic moment generating function $\Lambda: C_b(\Omega\times\mathcal{R})\rightarrow\mathbb{R}$ given by
\begin{eqnarray*}
\Lambda(f)&=&\lim_{n\rightarrow\infty}\frac{1}{n}\log E_o^{\omega}\left[\mathrm{e}^{n\langle f,\nu_{n,X}\rangle}\right]\\
&=&\lim_{n\rightarrow\infty}\frac{1}{n}\log E_o^{\omega}\left[\exp\left(\sum_{k=0}^{n-1} f(T_{X_k}\omega,X_{k+1}-X_k)\right)\right]
\end{eqnarray*} where $C_b$ denotes the space of bounded continuous functions.
\begin{theorem}\label{LMGF}
Suppose there exists an $\alpha>0$ such that (\ref{kimimvarki}) holds for each $z\in\mathcal{R}$. Then, the following are true:
\begin{description}
\item[Lower bound] For $\mathbb{P}$-a.e.\ $\omega$,
\begin{align*}
&\liminf_{n\rightarrow\infty}\frac{1}{n}\log E_o^{\omega}\left[\exp\left(\sum_{k=0}^{n-1} f(T_{X_k}\omega,X_{k+1}-X_k)\right)\right]\\
&\geq\sup_{\mu\in M_{1,s}^{\ll}(\Omega\times\mathcal{R})}\int\sum_{z\in\mathcal{R}}\mathrm{d}\mu(\omega,z)\left(f(\omega,z)-\log \frac{\mathrm{d}\mu(\omega,z)}{\mathrm{d}(\mu)^1(\omega)\pi(0,z)}\right)=:\Gamma(f).
\end{align*}
\item[Upper bound] For $\mathbb{P}$-a.e.\ $\omega$,
\begin{align*}
&\limsup_{n\rightarrow\infty}\frac{1}{n}\log E_o^{\omega}\left[\exp\left(\sum_{k=0}^{n-1} f(T_{X_k}\omega,X_{k+1}-X_k)\right)\right]\\
&\leq\inf_{F\in\mathcal{K}}\mathrm{ess}\sup_{\mathbb{P}}\log\sum_{z\in\mathcal{R}}\pi(0,z)\mathrm{e}^{f(\omega,z)+F(\omega,z)}=:\Lambda(f).
\end{align*}
\item[Equivalence of the bounds] For every $\epsilon>0$, there exists an $F_\epsilon\in\mathcal{K}$ such that
\[\mathrm{ess}\sup_{\mathbb{P}}\log\sum_{z\in\mathcal{R}}\pi(0,z)\mathrm{e}^{f(\omega,z)+F_\epsilon(\omega,z)}\leq\Gamma(f)+\epsilon.\]
\end{description}Thus, $\Lambda(f)\leq\Gamma(f)$. This clearly implies the existence of the logarithmic moment generating function.
\end{theorem}
Subsection \ref{LMGFsubsection} is devoted to the proof of Theorem \ref{LMGF}. After that, proving Theorem \ref{level2LDP} is easy: the LDP lower bound is obtained by a change of measure argument, and the LDP upper bound is a standard result since $M_1(\Omega\times\mathcal{R})$ is compact. Details are given in Subsection \ref{LDPproof}.

In our proofs, we will make frequent use of
\begin{lemma}[Kozlov \cite{Kozlov}]\label{Kozlov}
If an environment kernel $\hat{\pi}$ satisfies $\hat{\pi}(\cdot,z)>0$ $\mathbb{P}$-a.s.\ for each $z\in U$, and if there exists a $\hat{\pi}$-invariant probability measure $\mathbb{Q}\ll\mathbb{P}$, then the following hold:
\begin{itemize}
\item[(a)] The measures $\mathbb{P}$ and $\mathbb{Q}$ are in fact mutually absolutely continuous.
\item[(b)] The environment Markov chain with transition kernel $\hat{\pi}$ and initial distribution $\mathbb{Q}$ is stationary and ergodic.
\item[(c)] $\mathbb{Q}$ is the unique $\hat{\pi}$-invariant probability measure on $\Omega$ that is absolutely continuous relative to $\mathbb{P}$.
\item[(d)] The following LLN is satisfied: \[P_o^{\hat{\pi}}\left(\lim_{n\rightarrow\infty}\frac{X_n}{n}=\int\sum_{z\in\mathcal{R}}\hat{\pi}(\omega,z)z\;\mathrm{d}\mathbb{Q}\right) = 1.\]
\end{itemize}
\end{lemma}

\subsection{Logarithmic moment generating function}\label{LMGFsubsection}

\subsubsection{Lower bound}\label{pourum}

This is a standard change of measure argument. For any environment kernel $\hat{\pi}$ as in Definition \ref{ortamkeli},
\begin{align*}
&E_o^{\omega}\left[\exp\left(\sum_{k=0}^{n-1} f(T_{X_k}\omega,X_{k+1}-X_k)\right)\right]\\
&=E_o^{\hat{\pi},\omega}\left[\exp\left(\sum_{k=0}^{n-1} f(T_{X_k}\omega,X_{k+1}-X_k)\right)\,\frac{\mathrm{d}P_o^\omega}{\mathrm{d}P_o^{\hat{\pi},\omega}}\right]\\
&=E_o^{\hat{\pi},\omega}\left[\exp\left(\sum_{k=0}^{n-1}f(T_{X_k}\omega,X_{k+1}-X_k)-\log\frac{\hat{\pi}(T_{X_k}\omega,X_{k+1}-X_k)}{\pi(X_k,X_{k+1})}\right)\right].
\end{align*} If $\hat{\pi}(\cdot,z)>0$ $\mathbb{P}$-a.s.\ for each $z\in U$, and if there exists a $\phi\in L^1(\mathbb{P})$ such that $\phi\,\mathrm{d}\mathbb{P}$ is an invariant probability measure for the kernel $\hat{\pi}$, i.e., if $$\phi(\omega)=\sum_{z\in\mathcal{R}}\phi(T_{-z}\omega)\hat{\pi}(T_{-z}\omega,z)$$ for $\mathbb{P}$-a.e.\ $\omega$, then it follows from Lemma \ref{Kozlov} that $\phi\,\mathrm{d}\mathbb{P}$ is in fact an ergodic invariant measure for $\hat{\pi}$. By Jensen's inequality,
\begin{align}
&\liminf_{n\rightarrow\infty}\frac{1}{n}\log E_o^{\omega}\left[\exp\left(\sum_{k=0}^{n-1} f(T_{X_k}\omega,X_{k+1}-X_k)\right)\right]\nonumber\\
&\geq\liminf_{n\rightarrow\infty}E_o^{\hat{\pi},\omega}\left[\frac{1}{n}\sum_{k=0}^{n-1}f(T_{X_k}\omega,X_{k+1}-X_k)-\log\frac{\hat{\pi}(T_{X_k}\omega,X_{k+1}-X_k)}{\pi(X_k,X_{k+1})}\right]\nonumber\\
&=\int\sum_{z\in\mathcal{R}}\hat{\pi}(\omega,z)\left(f(\omega,z)-\log \frac{\hat{\pi}(\omega,z)}{\pi(0,z)}\right)\phi(\omega)\mathrm{d}\mathbb{P}=:H_f(\hat{\pi},\phi).\label{hakkariye}
\end{align} Therefore,
\begin{align}
&\liminf_{n\rightarrow\infty}\frac{1}{n}\log E_o^{\omega}\left[\exp\left(\sum_{k=0}^{n-1} f(T_{X_k}\omega,X_{k+1}-X_k)\right)\right]\nonumber\\
&\geq\sup_{(\hat{\pi},\phi)}\int\sum_{z\in\mathcal{R}}\hat{\pi}(\omega,z)\left(f(\omega,z)-\log \frac{\hat{\pi}(\omega,z)}{\pi(0,z)}\right)\phi(\omega)\mathrm{d}\mathbb{P}\label{stef}
\end{align} where the supremum is taken over the set of all $(\hat{\pi},\phi)$ pairs where $\hat{\pi}(\cdot,z)>0$ $\mathbb{P}$-a.s.\ for each $z\in U$ and $\phi\,\mathrm{d}\mathbb{P}$ is a $\hat{\pi}$-invariant probability measure. Note that there is a one-to-one correspondence between this set and $M_{1,s}^{\ll}(\Omega\times\mathcal{R})$. Hence, (\ref{stef}) is the desired lower bound.

Before proceeding with the upper bound, let us put (\ref{stef}) in a form that will turn out to be more convenient for showing the equivalence of the bounds. We start by giving a lemma.
\begin{lemma}\label{camilla}
For every $f\in C_b(\Omega\times\mathcal{R})$, $H_f$ (defined in (\ref{hakkariye})) has the following concavity property: For each $t\in(0,1)$ and any two pairs $(\hat{\pi}_1,\phi_1)$ and $(\hat{\pi}_2,\phi_2)$ where $\phi_i\,\mathrm{d}\mathbb{P}$ is $\hat{\pi}_i$-invariant (for $i=1,2$), define \[\gamma=\frac{t\phi_1}{t\phi_1+(1-t)\phi_2},\ \ \phi_3=t\phi_1+(1-t)\phi_2\ \ \mbox{and}\ \ \hat{\pi}_3=\gamma\hat{\pi}_1+(1-\gamma)\hat{\pi}_2.\] Then, $\phi_3\,\mathrm{d}\mathbb{P}$ is $\hat{\pi}_3$-invariant and 
\begin{equation}
H_f(\hat{\pi}_3,\phi_3)\geq tH_f(\hat{\pi}_1,\phi_1)+(1-t)H_f(\hat{\pi}_2,\phi_2).\label{yozgat}
\end{equation}
\end{lemma}

\begin{proof}
For any $t\in(0,1)$, it follows from the definitions and the assumptions in the statement of the lemma that $\mathbb{P}$-a.s.
\begin{align*}
\sum_{z\in\mathcal{R}}\phi_3(T_{-z}\omega)\hat{\pi}_3(T_{-z}\omega,z)&=\sum_{z\in\mathcal{R}}\phi_3(T_{-z}\omega)\gamma(T_{-z}\omega)\hat{\pi}_1(T_{-z}\omega,z)+\sum_{z\in\mathcal{R}}\phi_3(T_{-z}\omega)(1-\gamma(T_{-z}\omega))\hat{\pi}_2(T_{-z}\omega,z)\\
&=\ t\sum_{z\in\mathcal{R}}\phi_1(T_{-z}\omega)\hat{\pi}_1(T_{-z}\omega,z)+(1-t)\sum_{z\in\mathcal{R}}\phi_2(T_{-z}\omega)\hat{\pi}_2(T_{-z}\omega,z)\\
&=\ t\phi_1(\omega)+(1-t)\phi_2(\omega)=\phi_3(\omega).
\end{align*} In words, $\phi_3\,\mathrm{d}\mathbb{P}$ is $\hat{\pi}_3$-invariant. Finally,
\begin{align*}
H_f(\hat{\pi}_3,\phi_3)&=\int\sum_{z\in\mathcal{R}}\hat{\pi}_3(\omega,z)\left(f(\omega,z)-\log \frac{\hat{\pi}_3(\omega,z)}{\pi(0,z)}\right)\phi_3(\omega)\mathrm{d}\mathbb{P}\\&\geq\int\gamma(\omega)\sum_{z\in\mathcal{R}}\hat{\pi}_1(\omega,z)\left(f(\omega,z)-\log \frac{\hat{\pi}_1(\omega,z)}{\pi(0,z)}\right)\phi_3(\omega)\mathrm{d}\mathbb{P}\\&\ \ \ +\int(1-\gamma(\omega))\sum_{z\in\mathcal{R}}\hat{\pi}_2(\omega,z)\left(f(\omega,z)-\log \frac{\hat{\pi}_2(\omega,z)}{\pi(0,z)}\right)\phi_3(\omega)\mathrm{d}\mathbb{P}\\&=\ t\int\sum_{z\in\mathcal{R}}\hat{\pi}_1(\omega,z)\left(f(\omega,z)-\log \frac{\hat{\pi}_1(\omega,z)}{\pi(0,z)}\right)\phi_1(\omega)\mathrm{d}\mathbb{P}\\&\ \ \ +(1-t)\int\sum_{z\in\mathcal{R}}\hat{\pi}_2(\omega,z)\left(f(\omega,z)-\log \frac{\hat{\pi}_2(\omega,z)}{\pi(0,z)}\right)\phi_2(\omega)\mathrm{d}\mathbb{P}\\&=\ tH_f(\hat{\pi}_1,\phi_1)+(1-t)H_f(\hat{\pi}_2,\phi_2)
\end{align*} where the second line is obtained by applying Jensen's inequality to the integrand.
\end{proof}
Going back to the argument, let $\hat{\pi}_1(\omega,z):={1}/{(2d)}$ for each $z\in U$ and $\phi_1(\omega):=1$ for $\mathbb{P}$-a.e.\ $\omega$. An easy computation gives $H_f(\hat{\pi}_1,\phi_1)>-\infty$.  Take any pair $(\hat{\pi}_2,\phi_2)$ such that $\phi_2\,\mathrm{d}\mathbb{P}$ is $\hat{\pi}_2$-invariant and $H_f(\hat{\pi}_2,\phi_2)>-\infty$. For any $t\in(0,1)$, define $(\hat{\pi}_3,\phi_3)$ as in Lemma \ref{camilla} and see that $\hat{\pi}_3(\omega,z)>0$ $\mathbb{P}$-a.s.\ for each $z\in U$. Recalling (\ref{yozgat}), note that $H_f(\hat{\pi}_3,\phi_3)\geq (1-t)H_f(\hat{\pi}_2,\phi_2) + O(t)$. Since one can take $t$ arbitrarily small, the value of (\ref{stef}) does not change if the supremum there is taken over the set of all $(\hat{\pi},\phi)$ pairs where $\phi\,\mathrm{d}\mathbb{P}$ is a $\hat{\pi}$-invariant probability measure, dropping the positivity condition on $\hat{\pi}$. Finally, decouple $\hat{\pi}$ and $\phi$, and express the lower bound $\Gamma(f)$ as
\begin{equation}
\sup_{\phi}\sup_{\hat{\pi}}\inf_{h}\int\sum_{z\in\mathcal{R}}\hat{\pi}(\omega,z)\left(f(\omega,z)-\log \frac{\hat{\pi}(\omega,z)}{\pi(0,z)}+h(\omega)-h(T_z\omega)\right)\phi\,\mathrm{d}\mathbb{P}\label{putinh}
\end{equation} where the suprema are over all probability densities and all environment kernels, and the infimum is over all bounded measurable functions. This is due to the observation that if $\phi\,\mathrm{d}\mathbb{P}$ is not $\hat{\pi}$-invariant, then there exists a bounded measurable function $h:\Omega\to\mathbb{R}$ that satisfies \[\int\sum_{z\in\mathcal{R}}\hat{\pi}(\omega,z)\left(h(\omega)-h(T_z\omega)\right)\phi(\omega)\mathrm{d}\mathbb{P}\neq0,\] and taking scalar multiples of $h$ shows that the infimum in (\ref{putinh}) is $-\infty$.

\subsubsection{Upper bound}

Fix $f\in C_b(\Omega\times\mathcal{R})$. For any $F\in\mathcal{K}$, define \[K(F):=\mathrm{ess}\sup_{\mathbb{P}}\log\sum_{z\in\mathcal{R}}\pi(0,z)\mathrm{e}^{f(\omega,z)+F(\omega,z)}.\] Then, $\mathbb{P}$-a.s.
\begin{align}
&E_o^{\omega}\left[\left.\mathrm{e}^{f(T_{X_{n-1}}\omega, X_n-X_{n-1})+F(T_{X_{n-1}}\omega, X_n-X_{n-1})}\right|X_{n-1}\right]\label{itelebabam}\\
&=\sum_{z\in\mathcal{R}}\pi(X_{n-1},X_{n-1}+z)\mathrm{e}^{f(T_{X_{n-1}}\omega,z)+F(T_{X_{n-1}}\omega,z)}\nonumber\\
&\leq\mathrm{e}^{K(F)}.\nonumber
\end{align}
Taking conditional expectations and iterating (\ref{itelebabam}), one sees that $\mathbb{P}$-a.s. \[E_o^{\omega}\left[\exp\left(\sum_{k=0}^{n-1}f(T_{X_k}\omega,X_{k+1}-X_k)+F(T_{X_k}\omega,X_{k+1}-X_k)\right)\right]\leq\mathrm{e}^{nK(F)}.\] At this point, for any $\epsilon>0$, use Lemma \ref{GRR} (stated below) to write \[E_o^{\omega}\left[\exp\left(-c_\epsilon-n\epsilon+\sum_{k=0}^{n-1}f(T_{X_k}\omega,X_{k+1}-X_k)\right)\right]\leq\mathrm{e}^{nK(F)}\] where $c_\epsilon = c_\epsilon(\omega)$ is some constant. Arrange the terms to obtain \[\frac{1}{n}\log E_o^{\omega}\left[\exp\left(\sum_{k=0}^{n-1} f(T_{X_k}\omega,X_{k+1}-X_k)\right)\right]\leq K(F)+\epsilon+\frac{c_\epsilon}{n}.\] Let $n\to\infty,\ \epsilon\to 0$, and take infimum over $F\in\mathcal{K}$. This is the desired upper bound.

\begin{lemma}\label{GRR}
For every $F\in\mathcal{K}$, $\epsilon>0$ and $\mathbb{P}$-a.e.\ $\omega$, there exists $c_\epsilon= c_\epsilon(\omega)\geq 0$ such that for any sequence $(x_{k})_{k=0}^n$ with $x_o=0$ and $x_{k+1}-x_k\in\mathcal{R}$, \[\left|\sum_{k=0}^{n-1}F(T_{x_k}\omega,x_{k+1}-x_k)\right| \leq c_\epsilon+n\epsilon\] for all $n\geq1$.
\end{lemma}

\begin{remark}\label{nilosmu}
Lemma \ref{GRR} is proved in Chapter 2 of \cite{jeffrey}. See Appendix B for a sketch of the proof. In his definition of class $\mathcal{K}$, Rosenbluth takes $F: \Omega\times U \to\mathbb{R}$. But, such functions uniquely extend to $\Omega\times\mathcal{R}$ by the closed loop condition in Definition \ref{K}, and Lemma \ref{GRR} remains to be valid without any extra work.
\end{remark}

\subsubsection{Equivalence of the bounds}

Consider a sequence $\left(\mathcal{E}_k\right)_{k\geq1}$ of finite $\sigma$-algebras such that
$\mathcal{B}=\sigma\left(\bigcup_k \mathcal{E}_k\right)$ and $\mathcal{E}_{k}\subset T_z\mathcal{E}_{k+1}$ for all $z\in\mathcal{R}$. Then, recall (\ref{putinh}) and see that $\Gamma(f)$ can be bounded below by
\begin{align}
&\sup_{\phi}\sup_{\hat{\pi}}\inf_{h}\int\sum_{z\in\mathcal{R}}\hat{\pi}(\omega,z)\left(f(\omega,z)-\log \frac{\hat{\pi}(\omega,z)}{\pi(0,z)}+h(\omega)-h(T_z\omega)\right)\phi\,\mathrm{d}\mathbb{P}\label{thingone}\\=&\sup_{\phi}\inf_{h}\sup_{\hat{\pi}}\int\sum_{z\in\mathcal{R}}\hat{\pi}(\omega,z)\left(f(\omega,z)-\log \frac{\hat{\pi}(\omega,z)}{\pi(0,z)}+h(\omega)-h(T_z\omega)\right)\phi\,\mathrm{d}\mathbb{P}\label{thingtwo}\\=&\sup_{\phi}\inf_{h}\sup_{\hat{\pi}}\int\sum_{z\in\mathcal{R}}\left[v(\omega,z)-\log\hat{\pi}(\omega,z)\right]\hat{\pi}(\omega,z)\phi\,\mathrm{d}\mathbb{P}\label{thingthree}\\=&\sup_{\phi}\inf_{h}\int\sup_{\hat{\pi}(\omega,\cdot)}\left(\sum_{z\in\mathcal{R}}[v(\omega,z)-\log\hat{\pi}(\omega,z)]\hat{\pi}(\omega,z)\right)\phi\,\mathrm{d}\mathbb{P}\label{thingfour}\\=&\sup_{\phi}\inf_{h}\int\left(\log\sum_{z\in\mathcal{R}}\mathrm{e}^{v(\omega,z)}\right)\phi\,\mathrm{d}\mathbb{P}\label{thingfive}\\=&\inf_{h}\sup_{\phi}\int\left(\log\sum_{z\in\mathcal{R}}\mathrm{e}^{v(\omega,z)}\right)\phi\,\mathrm{d}\mathbb{P}\label{thingsix}\\=&\inf_{h}\mathrm{ess}\sup_{\mathbb{P}}\log\sum_{z\in\mathcal{R}}\mathrm{e}^{v(\omega,z)}.\label{thingseven}
\end{align}
Explanation: In (\ref{thingone}), the first supremum is taken over $\mathcal{E}_k$-measurable probability densities, the second supremum is over $\mathcal{E}_k$-measurable environment kernels and the infimum is over bounded $\mathcal{B}$-measurable functions. For each $\phi$, the second supremum in (\ref{thingone}) is over a compact set, the integral is concave and continuous in $\hat{\pi}$ and affine (hence convex) in $h$. Thus, one can apply the minimax theorem of Ky Fan \cite{KyFan} and obtain (\ref{thingtwo}). The integral in (\ref{thingtwo}) can be evaluated in two steps by first taking a conditional expectation with respect to $\mathcal{E}_k$. This gives (\ref{thingthree}) where \[v(\omega,z):=\mathbb{E}\left[\log \pi(0,z) + f(\omega,z) + h(\omega) - h(T_z\omega)\left|\mathcal{E}_k\right.\right].\] The integrand in (\ref{thingthree}) is a local function of $\hat{\pi}(\omega,\cdot)$, therefore the supremum can be taken inside the integral to obtain (\ref{thingfour}). Apply the method of Lagrange multipliers and see that the supremum in (\ref{thingfour}) is attained at \[\hat{\pi}(\omega,z)=\frac{\mathrm{e}^{v(\omega,z)}}{\sum_{z'\in\mathcal{R}}\mathrm{e}^{v(\omega,z')}}.\] Plugging this back in (\ref{thingfour}) gives (\ref{thingfive}). The integral in (\ref{thingfive}) is convex in $h$, and affine (hence concave) and continuous in $\phi$. Plus, the supremum is taken over a compact set. Thus, one can again apply the minimax theorem of Ky Fan \cite{KyFan} and arrive at (\ref{thingsix}) which is clearly equal to (\ref{thingseven}).

Let us proceed with the proof. (\ref{thingseven}) implies that $\forall\epsilon >0$ and $k\geq1$, there exists an $h_{k,\epsilon}$ that satisfies
\begin{equation}
\log\sum_{z\in\mathcal{R}}\exp \mathbb{E}\left[\log \pi(0,z) + f(\omega,z) + h_{k,\epsilon}(\omega) - h_{k,\epsilon}(T_z\omega)\left|\mathcal{E}_k\right.\right]\leq\Gamma(f) + \epsilon\label{coklugot}
\end{equation} for $\mathbb{P}$-a.e.\ $\omega$. Therefore,
\begin{equation}
\mathbb{E}\left[h_{k,\epsilon}(\omega) - h_{k,\epsilon}(T_z\omega)\left|\mathcal{E}_k\right.\right]\leq\mathbb{E}\left[-\log \pi(0,z)\left|\mathcal{E}_k\right.\right]+\|f\|_{\infty}+\Gamma(f)+\epsilon\label{yarinbitersekral}
\end{equation} for each $z\in\mathcal{R}$. Define $F_{k,\epsilon}:\Omega\times\mathcal{R}\to\mathbb{R}$ by $F_{k,\epsilon}(\omega,z):=\mathbb{E}\left[h_{k,\epsilon}(\omega) - h_{k,\epsilon}(T_z\omega)\left|\mathcal{E}_{k-1}\right.\right]$. Then,
\begin{equation}
F_{k,\epsilon}(\omega,z)\leq\mathbb{E}\left[-\log \pi(0,z)\left|\mathcal{E}_{k-1}\right.\right]+\|f\|_{\infty}+\Gamma(f)+\epsilon\label{kirmizigul}
\end{equation} holds $\mathbb{P}$-a.s.\ for each $z\in\mathcal{R}$. Also, note that
\begin{align*}
-\mathbb{E}\left[h_{k,\epsilon}(\omega) - h_{k,\epsilon}(T_z\omega)\left|T_{-z}\mathcal{E}_k\right.\right]&=-\mathbb{E}\left[h_{k,\epsilon}(T_{-z}\omega) - h_{k,\epsilon}(\omega)\left|\mathcal{E}_k\right.\right](T_z\cdot)\\
&=\mathbb{E}\left[h_{k,\epsilon}(\omega) - h_{k,\epsilon}(T_{-z}\omega)\left|\mathcal{E}_k\right.\right](T_z\cdot)\\
&\leq\mathbb{E}\left[-\log \pi(0,-z)\left|\mathcal{E}_k\right.\right](T_z\cdot)+\|f\|_{\infty}+\Gamma(f)+\epsilon\\
&=\mathbb{E}\left[-\log \pi(z,0)\left|T_{-z}\mathcal{E}_k\right.\right]+\|f\|_{\infty}+\Gamma(f)+\epsilon
\end{align*} where the inequality follows from (\ref{yarinbitersekral}). Since $\mathcal{E}_{k-1}\subset T_{-z}\mathcal{E}_k$, taking conditional expectation with respect to $\mathcal{E}_{k-1}$ gives \[-F_{k,\epsilon}(\omega,z)\leq\mathbb{E}\left[-\log \pi(z,0)\left|\mathcal{E}_{k-1}\right.\right]+\|f\|_{\infty}+\Gamma(f)+\epsilon.\] Recall (\ref{kirmizigul}) and deduce that \[\left|F_{k,\epsilon}(\omega,z)\right|\leq\mathbb{E}\left[-\log \pi(0,z)\left|\mathcal{E}_{k-1}\right.\right]+\mathbb{E}\left[-\log \pi(z,0)\left|\mathcal{E}_{k-1}\right.\right]+\|f\|_{\infty}+\Gamma(f)+\epsilon.\] This implies by (\ref{kimimvarki}) that $\left(F_{k,\epsilon}(\cdot,z)\right)_{k\geq1}$ is uniformly bounded in $L^{d+\alpha}(\mathbb{P})$ for each $z\in\mathcal{R}$. Passing to a subsequence if necessary, $F_{k,\epsilon}(\cdot,z)$ converges weakly to a limit $F_{\epsilon}(\cdot,z)\in L^{d+\alpha}(\mathbb{P})$.

For $j\geq1$ and any sequence $(x_{i})_{i=0}^n$ in $\mathbb{Z}^d$ such that $x_{i+1}-x_i\in\mathcal{R}$ and $x_0=x_n$, 
\begin{align}
&\mathbb{E}\left(\left.\sum_{i=0}^{n-1}F_{\epsilon}(T_{x_i}\omega,x_{i+1}-x_i)\right|\mathcal{E}_j\right)\nonumber\\
&=\sum_{i=0}^{n-1}\mathbb{E}\left(\left.\lim_{k\to\infty}F_{k,\epsilon}(T_{x_i}\omega,x_{i+1}-x_i)\right|\mathcal{E}_j\right)\nonumber\\
&=\sum_{i=0}^{n-1}\lim_{k\to\infty}\mathbb{E}\left(\left.F_{k,\epsilon}(T_{x_i}\omega,x_{i+1}-x_i)\right|\mathcal{E}_j\right)\nonumber\\
&=\sum_{i=0}^{n-1}\lim_{k\to\infty}\mathbb{E}\left(\left.\mathbb{E}\left[h_{k,\epsilon}(\omega) - h_{k,\epsilon}(T_{x_{i+1}-x_i}\omega)\left|\mathcal{E}_{k-1}\right.\right](T_{x_i}\omega)\right|\mathcal{E}_j\right)\nonumber\\
&=\sum_{i=0}^{n-1}\lim_{k\to\infty}\mathbb{E}\left(\left.\mathbb{E}\left[h_{k,\epsilon}(T_{x_i}\omega) - h_{k,\epsilon}(T_{x_{i+1}}\omega)\left|T_{-x_i}\mathcal{E}_{k-1}\right.\right]\right|\mathcal{E}_j\right)\nonumber\\
&=\sum_{i=0}^{n-1}\lim_{k\to\infty}\mathbb{E}\left(\left.h_{k,\epsilon}(T_{x_i}\omega) - h_{k,\epsilon}(T_{x_{i+1}}\omega)\right|\mathcal{E}_j\right)\label{dursunmus}\\
&=\lim_{k\to\infty}\mathbb{E}\left(\left.\sum_{i=0}^{n-1}\left(h_{k,\epsilon}(T_{x_i}\omega) - h_{k,\epsilon}(T_{x_{i+1}}\omega)\right)\right|\mathcal{E}_j\right)=0\nonumber
\end{align} holds $\mathbb{P}$-a.s., where (\ref{dursunmus}) follows from the fact that $\mathcal{E}_j\subset T_{-x_i}\mathcal{E}_{k-1}$ whenever $k$ is large enough. Therefore, $\sum_{i=0}^{n-1}F_{\epsilon}(T_{x_i}\omega,x_{i+1}-x_i)=0$ for $\mathbb{P}$-a.e.\ $\omega$, and $F_{\epsilon}:\Omega\times\mathcal{R}\to\mathbb{R}$ satisfies the closed loop condition given in Definition \ref{K}. We already know that it satisfies the moment condition, and it is also clearly mean zero. Hence, $F_{\epsilon}\in\mathcal{K}$.

Since $\mathbb{E}\left[\log \pi(0,z) + f(\omega,z)\left|\mathcal{E}_{k-1}\right.\right]$ is an $L^{d+\alpha}(\mathbb{P})$-bounded martingale, it converges to $\log \pi(0,z) + f(\cdot,z)$ in $L^{d+\alpha}(\mathbb{P})$. Therefore, \[\mathcal{L}_{k,\epsilon}(\cdot,z):=\mathbb{E}\left[\log \pi(0,z) + f(\omega,z)\left|\mathcal{E}_{k-1}\right.\right]+F_{k,\epsilon}(\cdot,z)\] converges weakly in $L^{d+\alpha}(\mathbb{P})$ to $\log \pi(0,z) + f(\cdot,z)+F_{\epsilon}(\cdot,z)$. By Mazur's theorem (see \cite{Rudin}), there exist $\mathcal{L}_{k,\epsilon}':\Omega\times\mathcal{R}\to\mathbb{R}$ for $k\geq1$ such that $\mathcal{L}_{k,\epsilon}'(\cdot,z)$ converges strongly in $L^{d+\alpha}(\mathbb{P})$ to $\log \pi(0,z) + f(\cdot,z)+F_{\epsilon}(\cdot,z)$ for each $z\in\mathcal{R}$ and $\mathcal{L}_{k,\epsilon}'$ is a convex combination of $\{\mathcal{L}_{1,\epsilon},\mathcal{L}_{2,\epsilon},\ldots,\mathcal{L}_{k,\epsilon}\}$. Passing to a further subsequence, $\mathcal{L}_{k,\epsilon}'(\cdot,z)$ converges $\mathbb{P}$-a.s.\ to $\log \pi(0,z) + f(\cdot,z)+F_{\epsilon}(\cdot,z)$. Take conditional expectation of both sides of (\ref{coklugot}) with respect to $\mathcal{E}_{k-1}$ and use Jensen's inequality to write \[\log\sum_{z\in\mathcal{R}}\exp\left(\mathbb{E}\left[\log \pi(0,z) + f(\omega,z)\left|\mathcal{E}_{k-1}\right.\right]+F_{k,\epsilon}(\cdot,z)\right)\leq\Gamma(f) + \epsilon.\] Again by Jensen's inequality, $\log\sum_{z\in\mathcal{R}}\exp\left(\mathcal{L}_{k,\epsilon}'(\cdot,z)\right)\leq\Gamma(f) + \epsilon$. Taking $k\to\infty$ gives \[\log\sum_{z\in\mathcal{R}}\pi(0,z)\mathrm{e}^{f(\omega,z)+F_{\epsilon}(\omega,z)}\leq\Gamma(f) + \epsilon\] for $\mathbb{P}$-a.e.\ $\omega$. Theorem \ref{LMGF} is proved.

\subsection{Large deviation principle}\label{LDPproof}

Putting together (\ref{level2ratetilde}) and Theorem \ref{LMGF}, one sees that
\begin{align*}
\Lambda(f)&=\sup_{\mu\in M_{1,s}^{\ll}(\Omega\times\mathcal{R})}\int\sum_{z\in\mathcal{R}}\mathrm{d}\mu(\omega,z)\left(f(\omega,z)-\log \frac{\mathrm{d}\mu(\omega,z)}{\mathrm{d}(\mu)^1(\omega)\pi(0,z)}\right)\\
&=\sup_{\mu\in M_{1,s}^{\ll}(\Omega\times\mathcal{R})}\left\{\left\langle f,\mu\right\rangle - \mathfrak{I}(\mu)\right\}\\
&=\sup_{\mu\in M_1(\Omega\times\mathcal{R})}\left\{\left\langle f,\mu\right\rangle - \mathfrak{I}(\mu)\right\}\\
&=\mathfrak{I}^*(f),
\end{align*} the Fenchel-Legendre transform of $\mathfrak{I}$. Therefore, $\mathfrak{I}^{**}=\Lambda^*$.

Since $M_1(\Omega\times\mathcal{R})$ is compact,
\begin{displaymath}
\limsup_{n\rightarrow\infty}\frac{1}{n}\log P_o^{\omega}(\nu_{n,X}\in C)\leq-\inf_{\mu\in C}\Lambda^*(\mu)=-\inf_{\mu\in C}\mathfrak{I}^{**}(\mu)
\end{displaymath}
for $\mathbb{P}$-a.e.\ $\omega$ and any closed subset $C$ of $M_1(\Omega\times\mathcal{R})$. (See Theorem 4.5.3 of \cite{DemboZeitouni}.)

To conclude the proof of Theorem \ref{level2LDP}, one needs to obtain the LDP lower bound. Note that, for any open subset $G$ of $M_1(\Omega\times\mathcal{R})$, $\inf_{\nu\in G}\mathfrak{I}^{**}(\nu)=\inf_{\nu\in G}\mathfrak{I}(\nu)$. (See \cite{Rockafellar}, page 104.) Therefore, it suffices to show that, for any $\mu\in M_{1,s}^{\ll}(\Omega\times\mathcal{R})$, any open set $O$ containing $\mu$ and $\mathbb{P}$-a.e.\ $\omega$,
\begin{equation}
\liminf_{n\rightarrow\infty}\frac{1}{n}\log P_o^{\omega}(\nu_{n,X}\in O)\geq-\mathfrak{I}(\mu).\label{LB}
\end{equation} 
Take the pair \[(\hat{\pi},\phi):=\left(\frac{\mathrm{d}\mu}{\mathrm{d}(\mu)^1},\frac{\mathrm{d}(\mu)^1}{\mathrm{d}\mathbb{P}}\right)\] corresponding to a given $\mu\in M_{1,s}^{\ll}(\Omega\times\mathcal{R})$. Then, $\phi\in L^1(\mathbb{P})$, $\phi\,\mathrm{d}\mathbb{P}$ is a $\hat{\pi}$-invariant probability measure, and $\hat{\pi}(\cdot,z)>0$ $\mathbb{P}$-a.s.\ for each $z\in U$. With this notation, (\ref{LB}) becomes \[\liminf_{n\rightarrow\infty}\frac{1}{n}\log P_o^{\omega}(\nu_{n,X}\in O)\geq-\int_{\Omega}\sum_{z\in\mathcal{R}}\hat{\pi}(\omega,z)\log\frac{\hat{\pi}(\omega,z)}{\pi(0,z)}\phi(\omega)\mathrm{d}\mathbb{P}.\] Recall Definition \ref{ortamkeli} and introduce a new measure $R_o^{\hat{\pi},\omega}$ by setting \[\mathrm{d}R_o^{\hat{\pi},\omega}:=\frac{\one_{\nu_{n,X}\in O}}{P_o^{\hat{\pi},\omega}(\nu_{n,X}\in O)}\,\mathrm{d}P_o^{\hat{\pi},\omega}.\] Then,
\begin{align*}
\liminf_{n\rightarrow\infty}\frac{1}{n}\log P_o^\omega(\nu_{n,X}\in O)=&\liminf_{n\rightarrow\infty}\frac{1}{n}\log E_o^{\hat{\pi},\omega}\left[\one_{\nu_{n,X}\in O}\,\frac{\mathrm{d}P_o^\omega}{\mathrm{d}P_o^{\hat{\pi},\omega}}\right]\\
=&\liminf_{n\rightarrow\infty}\frac{1}{n}\left(\log P_o^{\hat{\pi},\omega}(\nu_{n,X}\in O)+\log \int\frac{\mathrm{d}P_o^\omega}{\mathrm{d}P_o^{\hat{\pi},\omega}}\mathrm{d}R_o^{\hat{\pi},\omega}\right)\\
\geq&\liminf_{n\rightarrow\infty}\frac{1}{n}\left(\log P_o^{\hat{\pi},\omega}(\nu_{n,X}\in O)- \int\log\frac{\mathrm{d}P_o^{\hat{\pi},\omega}}{\mathrm{d}P_o^\omega}\mathrm{d}R_o^{\hat{\pi},\omega}\right)\\
=&\liminf_{n\rightarrow\infty}\frac{1}{n}\left(\log P_o^{\hat{\pi},\omega}(\nu_{n,X}\in O)-\frac{1}{P_o^{\hat{\pi},\omega}(\nu_{n,X}\in O)} E_o^{\hat{\pi},\omega}\left[\one_{\nu_{n,X}\in O}\,\log\frac{\mathrm{d}P_o^{\hat{\pi},\omega}}{\mathrm{d}P_o^\omega}\right]\right)
\end{align*} where the third line uses Jensen's inequality. It follows from Lemma \ref{Kozlov} that $\lim_{n\rightarrow\infty}P_o^{\hat{\pi},\omega}(\nu_{n,X}\in O)=1$. Therefore,
\begin{align*}
\liminf_{n\rightarrow\infty}\frac{1}{n}\log P_o^\omega(\nu_{n,X}\in O)&\geq-\limsup_{n\rightarrow\infty}\frac{1}{n}E_o^{\hat{\pi},\omega}\left[\one_{\nu_{n,X}\in O}\,\log\frac{\mathrm{d}P_o^{\hat{\pi},\omega}}{\mathrm{d}P_o^\omega}\right]\\
&=-\int_{\Omega}\sum_{z\in\mathcal{R}}\hat{\pi}(\omega,z)\log\frac{\hat{\pi}(\omega,z)}{\pi(0,z)}\phi(\omega)\mathrm{d}\mathbb{P}
\end{align*} again by Lemma \ref{Kozlov} and the $L^1$-ergodic theorem. Theorem \ref{level2LDP} is proved. Finally, note that the convexity of $\mathfrak{I}$ follows from an argument similar to the proof of Lemma \ref{camilla}.
\begin{remark}
$\mathfrak{I}^{**}$ is a good rate function since $M_1(\Omega\times\mathcal{R})$ is compact.
\end{remark}

\section{Contraction principle and the Ansatz for the minimizer}\label{birboyutadonus}

\begin{proof}[Proof of Corollary \ref{level1LDP}]
Recall (\ref{ximu}) and observe that \[\xi_{\nu_{n,X}}=\int\sum_{z\in\mathcal{R}}\mathrm{d}\nu_{n,X}(\omega,z)z=\frac{1}{n}\sum_{k=0}^{n-1}\left(X_{k+1}-X_k\right)=\frac{X_n-X_o}{n}.\] Therefore, as noted in Subsection \ref{results}, Corollary \ref{level1LDP} follows from Theorem \ref{level2LDP} by the contraction principle (see \cite{DemboZeitouni}), and the rate function is given by (\ref{level1rate}). 

In order to justify (\ref{level1ratetilde}), let us define $J:\mathbb{R}^d\rightarrow\mathbb{R}^+$ by $J(\xi)=\inf_{\mu\in A_\xi}\mathfrak{I}(\mu)$. We would like to show that $J\equiv I$. Since $\mathfrak{I}$ and $\mathfrak{I}^{**}$ are convex, $I$ and $J$ are convex functions on $\mathbb{R}^d$. Therefore, it suffices to show that $J^*\equiv I^*$. For any $\eta\in\mathbb{R}^d$, define $f_{\eta}\in C_b(\Omega\times\mathcal{R})$ by $f_{\eta}(\omega,z):=\langle z,\eta\rangle$. Recalling (\ref{ximu}),
\begin{eqnarray*}
I^*(\eta)&=&\sup_{\xi}\{\langle\eta,\xi\rangle - \inf_{\mu\in A_\xi}\mathfrak{I}^{**}(\mu)\}\\
&=&\sup_{\xi}\sup_{\mu\in A_\xi}\{\langle\eta,\xi_{\mu}\rangle - \mathfrak{I}^{**}(\mu)\}\\
&=&\sup_{\mu\in M_1(\Omega\times\mathcal{R})}\{\langle f_{\eta},\mu\rangle - \mathfrak{I}^{**}(\mu)\}\\
&=&\mathfrak{I}^{***}(f_{\eta})=\Lambda(f_{\eta}).
\end{eqnarray*}
Similarly, $J^*(\eta)=\mathfrak{I}^*(f_{\eta})=\Lambda(f_{\eta})$. We are done.
\end{proof}

\begin{proof}[Proof of Lemma \ref{lagrange}]
The rate function given by formula (\ref{level1ratetilde}) is
\begin{equation}
I(\xi)=\inf_{\mu\in A_\xi\cap M_{1,s}^{\ll}(\Omega\times\mathcal{R})}\int_{\Omega}\sum_{z\in\mathcal{R}} \mathrm{d}\mu(\omega,z)\log\frac{\mathrm{d}\mu(\omega,z)}{\mathrm{d}(\mu)^1(\omega)\pi(0,z)}.\label{budur}
\end{equation} Fix a $\xi=(\xi_1,\ldots,\xi_d)\in\mathbb{R}^d$ with $|\xi_1|+\cdots+|\xi_d|\leq B$. (Otherwise, the set $A_\xi$ is empty.) If there exists a $\mu_\xi\in A_\xi\cap M_{1,s}^{\ll}(\Omega\times\mathcal{R})$ such that \[\mathrm{d}\mu_\xi(\omega,z)=\mathrm{d}(\mu_\xi)^1(\omega) \pi(0,z)\mathrm{e}^{\langle\theta,z\rangle+F(\omega,z)+ r}\] for some $\theta\in\mathbb{R}^d$, $F\in\mathcal{K}$ and $r\in\mathbb{R}$, then for any $\nu\in A_\xi\cap M_{1,s}^{\ll}(\Omega\times\mathcal{R})$,
\begin{align*}
\mathfrak{I}(\nu)&=\int_{\Omega}\sum_{z\in\mathcal{R}} \mathrm{d}\nu(\omega,z)\log\frac{\mathrm{d}\nu(\omega,z)}{\mathrm{d}(\nu)^1(\omega)\pi(0,z)}\\
&=\int_{\Omega}\sum_{z\in\mathcal{R}} \mathrm{d}\nu(\omega,z)\log\frac{\mathrm{d}\nu(\omega,z)\mathrm{e}^{\langle\theta,z\rangle+F(\omega,z)+r}}{\mathrm{d}(\nu)^1(\omega)\pi(0,z)\mathrm{e}^{\langle\theta,z\rangle+F(\omega,z)+r}}\\
&=\int_{\Omega}\sum_{z\in\mathcal{R}} \mathrm{d}\nu(\omega,z)\left(\langle\theta,z\rangle+F(\omega,z)+r+\log\frac{\mathrm{d}\nu(\omega,z)\;\mathrm{d}(\mu_\xi)^1(\omega)}{\mathrm{d}(\nu)^1(\omega)\;\mathrm{d}\mu_\xi(\omega,z)}\right)\\ &=\langle\theta,\xi\rangle+r+\int_{\Omega}\sum_{z\in\mathcal{R}} \mathrm{d}\nu(\omega,z)F(\omega,z)+\int_{\Omega}\sum_{z\in\mathcal{R}} \mathrm{d}\nu(\omega,z)\log\frac{\mathrm{d}\nu(\omega,z)\;\mathrm{d}(\mu_\xi)^1(\omega)}{\mathrm{d}(\nu)^1(\omega)\;\mathrm{d}\mu_\xi(\omega,z)}.
\end{align*}
Under the Markov kernel $\frac{\mathrm{d}\nu}{\mathrm{d}(\nu)^1}$ with invariant measure $(\nu)^1$, $\mathbb{P}$-a.s.\[\lim_{n\rightarrow\infty}\frac{1}{n}\sum_{k=0}^{n-1}F(T_{X_k}\omega,X_{k+1}-X_k)=\int_{\Omega}\sum_{z\in\mathcal{R}} \mathrm{d}\nu(\omega,z)F(\omega,z)\] by Lemma \ref{Kozlov} and the ergodic theorem. But, the same limit is $0$ by Lemma \ref{GRR}. Therefore,
\begin{equation}
\mathfrak{I}(\nu)=\langle\theta,\xi\rangle+r+\int_{\Omega}\sum_{z\in\mathcal{R}} \mathrm{d}\nu(\omega,z)\log\frac{\mathrm{d}\nu(\omega,z)\;\mathrm{d}(\mu_\xi)^1(\omega)}{\mathrm{d}(\nu)^1(\omega)\;\mathrm{d}\mu_\xi(\omega,z)}.
\label{sifirladik}
\end{equation}
By an application of Jensen's inequality, it is easy to see that the integral on the RHS of (\ref{sifirladik}) is nonnegative. Moreover, this integral is zero if and only if $\frac{\mathrm{d}\nu}{\mathrm{d}(\nu)^1}=\frac{\mathrm{d}\mu_\xi}{\mathrm{d}(\mu_\xi)^1}$ holds $(\nu)^1$-a.s.\ and hence $\mathbb{P}$-a.s.\ by Lemma \ref{Kozlov}. Since $(\mu_\xi)^1$ is the unique invariant measure of $\frac{\mathrm{d}\mu_\xi}{\mathrm{d}(\mu_\xi)^1}$ that is absolutely continuous relative to $\mathbb{P}$ (again by Lemma \ref{Kozlov}), $\mu_\xi$ is the unique minimizer of (\ref{budur}).
\end{proof}

\section{Nearest-neighbor walks on $\mathbb{Z}$}\label{vandiseksin}

In this section, we carry out the recipe given in Lemma \ref{lagrange} and prove Theorem \ref{explicitformulah} in the case of nearest-neighbor random walk on $\mathbb{Z}$ in a stationary and ergodic environment. As mentioned in Subsection \ref{results}, we assume that the following holds:
\begin{enumerate}
\item [(A1)] There exists an $\alpha>0$ such that $\int|\log\pi(0,\pm1)|^{1+\alpha}\mathrm{d}\mathbb{P}<\infty$.
\end{enumerate}

\begin{proof}[Proof of Lemma \ref{lifeisrandom} for nearest-neighbor walks on $\mathbb{Z}$]
Define $\zeta(r,\omega):=E_o^\omega\left[\mathrm{e}^{r\tau_1},\tau_1<\infty\right]$ for any $r\in\mathbb{R}$. Then, 
\begin{align*}
\lambda(r)&=\lim_{n\to\infty}\frac{1}{n}\log E_o^\omega\left[\mathrm{e}^{r\tau_n},\tau_n<\infty\right]=\lim_{n\to\infty}\frac{1}{n}\log\left(\prod_{k=0}^{n-1}E_k^\omega\left[\mathrm{e}^{r\tau_{k+1}},\tau_{k+1}<\infty\right]\right)\\&=\lim_{n\to\infty}\frac{1}{n}\sum_{k=0}^{n-1}\log\zeta(r,T_{k}\omega)=\mathbb{E}\left[\log\zeta(r,\cdot)\right]
\end{align*} by the ergodic theorem, where the limits hold for $\mathbb{P}$-a.e.\ $\omega$. If $\zeta(r,\omega)$ is finite, then
\begin{align}
\zeta(r,\omega)&=\pi(0,1)\mathrm{e}^r+\pi(0,-1)\mathrm{e}^r\zeta(r,T_{-1}\omega)\zeta(r,\omega),\nonumber\\
1&=\pi(0,1)\mathrm{e}^r\zeta(r,\omega)^{-1}+\pi(0,-1)\mathrm{e}^r\zeta(r,T_{-1}\omega).\label{masterof}
\end{align} Since $\pi(0,-1)>0$ holds $\mathbb{P}$-a.s., the set $\{\omega:\zeta(r,\omega)<\infty\}$ is $T$-invariant, and its $\mathbb{P}$-probability is $0$ or $1$. The function $r\mapsto\zeta(r,\omega)$ is strictly increasing. There exists an $r_c\geq0$ such that $\mathbb{P}\left(\omega: \zeta(r,\omega)<\infty\right)=1$ if $r<r_c$ and $\mathbb{P}\left(\omega: \zeta(r,\omega)=\infty\right)=1$ if $r>r_c$. When $r<r_c$,
\begin{align*}
E_o^\omega\left[\mathrm{e}^{r\tau_1},\tau_1<\infty\right]&\geq E_o^\omega\left[\mathrm{e}^{r\tau_1}, X_k=-k, X_{2k}=0, \tau_1<\infty\right]\\&=\mathrm{e}^{2rk}\left(\prod_{i=0}^{-k+1}\pi(i,i-1)\prod_{j=-k}^{-1}\pi(j,j+1)\right)E_o^\omega\left[\mathrm{e}^{r\tau_1},\tau_1<\infty\right]
\end{align*} for any $k\geq1$. Cancelling the $E_o^\omega\left[\mathrm{e}^{r\tau_1},\tau_1<\infty\right]$ term on both sides and taking logarithms give $$2rk + \sum_{i=0}^{-k+1}\log\pi(i,i-1) + \sum_{j=-k}^{-1}\log\pi(j,j+1) \leq 0.$$ Divide both sides by $k$, let $k\to\infty$, and see that $2r\leq-\mathbb{E}[\log\pi(0,-1)] - \mathbb{E}[\log\pi(0,1)]$ by the ergodic theorem. This, in combination with (A1), implies that $r_c<\infty.$

By (\ref{masterof}), $1\geq\pi(0,-1)\mathrm{e}^r\zeta(r,T_{-1}\omega)$ and $\log\zeta(r,T_{-1}\omega)\leq-\log\pi(0,-1)-r$. Thus, \begin{equation}\lambda(r)=\mathbb{E}[\log\zeta(r,\cdot)]\leq\int|\log\pi(0,-1)|\mathrm{d}\mathbb{P}-r<\infty\label{turran}\end{equation} for $r<r_c$, and also for $r=r_c$ by the monotone convergence theorem. 

It is easy to see that $r\mapsto\lambda(r)=\mathbb{E}[\log\zeta(r,\cdot)]$ is analytic on $(-\infty,r_c)$. Assumption (A1) ensures that the walk under $P_o^\omega$ is not deterministic, therefore
$$\lambda''(r)=\mathbb{E}\left[\frac{E_o^\omega\left[\tau_1^2\mathrm{e}^{r\tau_1},\tau_1<\infty\right]}{E_o^\omega\left[\mathrm{e}^{r\tau_1},\tau_1<\infty\right]} - \left(\frac{E_o^\omega\left[\tau_1\mathrm{e}^{r\tau_1},\tau_1<\infty\right]}{E_o^\omega\left[\mathrm{e}^{r\tau_1},\tau_1<\infty\right]}\right)^2\right]$$ is strictly positive by Jensen's inequality. Hence, $r\mapsto\lambda(r)$ is strictly convex on $(-\infty,r_c)$.
Recall that $\xi_c:=\lambda'(r_c-)^{-1}$. Use again the fact that the walk under $P_o^\omega$ is not deterministic, and write \[\xi_c^{-1}=\lambda'(r_c-)\geq\lambda'(0-)=\mathbb{E}\left(E_o^\omega[\left.\tau_1\right|\tau_1<\infty]\right)>1.\] We have proved half of Lemma \ref{lifeisrandom}, namely the statements involving $r\mapsto\lambda(r)$. Simply replace $\tau_n$ by $\bar{\tau}_{-n}$ to prove the other half of the lemma.

What remains to be shown is that the same $r_c$ works for $\lambda(\cdot)$ and $\bar{\lambda}(\cdot)$. This is proved in Appendix C.
\end{proof}

Let us start the construction. Note that $$\lim_{r\to-\infty}\lambda'(r)=\lim_{r\to-\infty}\mathbb{E}\left(\frac{E_o^\omega\left[\tau_1\mathrm{e}^{r\tau_1},\tau_1<\infty\right]}{E_o^\omega\left[\mathrm{e}^{r\tau_1},\tau_1<\infty\right]}\right)=1.$$ The map $r\mapsto\lambda'(r)$ is strictly increasing and, therefore, it is a bijection from $(-\infty,r_c)$ to $(1,\xi_c^{-1})$. In other words, for any $\xi\in(\xi_c,1)$, there is a unique $r=r(\xi)<r_c$ such that $\xi^{-1}=\lambda'(r)$.

Taking $r=r(\xi)$, recall (\ref{masterof}) and define an environment kernel $\hat{\pi}_r$ (in the sense of Definition \ref{ortamkeli}) by
\begin{equation}
\hat{\pi}_r(\omega,1):=\pi(0,1)\mathrm{e}^r\zeta(r,\omega)^{-1}\quad\mbox{and}\quad\hat{\pi}_r(\omega,-1):=\pi(0,-1)\mathrm{e}^r\zeta(r,T_{-1}\omega).\label{yyedin}
\end{equation}
For every $x\in\mathbb{Z}$, in order to simplify the notation, $P_x^{\hat{\pi}_r,\omega}, E_x^{\hat{\pi}_r,\omega}, P_x^{\hat{\pi}_r}$ and $E_x^{\hat{\pi}_r}$ are denoted by $P_x^{r,\omega}, E_x^{r,\omega}, P_x^r$ and $E_x^r$, respectively.
For $\mathbb{P}$-a.e.\ $\omega$ and any finite sequence $(x_{k})_{k=0}^n$ in $\mathbb{Z}$ such that $x_{k+1}-x_k\in U$ and $x_n=1$, it is easy to see that
\begin{equation}\label{kfer}
P_o^{r,\omega}(X_1=x_1, \ldots, X_n = x_n) = \mathrm{e}^{rn}\zeta(r,\omega)^{-1}P_o^\omega(X_1=x_1, \ldots, X_n = x_n).
\end{equation}

\begin{lemma}
$P_o^{r}(\tau_1<\infty)=1$.
\end{lemma}
\begin{proof}
For $\mathbb{P}$-a.e.\ $\omega$, $$P_o^{r,\omega}(\tau_1<\infty)=E_o^\omega[\mathrm{e}^{r\tau_1}\zeta(r,\omega)^{-1},\tau_1<\infty]=\zeta(r,\omega)^{-1}E_o^\omega[\mathrm{e}^{r\tau_1},\tau_1<\infty]=1$$ where the first equality follows from (\ref{kfer}).
\end{proof}


\begin{lemma}\label{shmeryahu}
$E_o^{r}[\tau_1]=\xi^{-1}<\infty$.
\end{lemma}
\begin{proof}
For any $s\in\mathbb{R}$ and $\mathbb{P}$-a.e.\ $\omega$, recall (\ref{kfer}) and observe that
\begin{align*}E_o^{r,\omega}[\mathrm{e}^{s\tau_1}]=E_o^{r,\omega}[\mathrm{e}^{s\tau_1},\tau_1<\infty]&=E_o^\omega[\mathrm{e}^{(r+s)\tau_1}\zeta(r,\omega)^{-1},\tau_1<\infty]\\&=\zeta(r+s,\omega)\zeta(r,\omega)^{-1}.
\end{align*}
Therefore, $\mathbb{E}\left(\log E_o^{r,\omega}[\mathrm{e}^{s\tau_1}]\right)=\lambda(r+s)-\lambda(r)<\infty$ by (\ref{turran}) whenever $r+s<r_c$, and \[E_o^{r}[\tau_1]=\left.\frac{\mathrm{d}}{\mathrm{d}s}\right|_{s=0}\!\!\!\!\!\mathbb{E}\left(\log E_o^{r,\omega}[\mathrm{e}^{s\tau_1}]\right)=\lambda'(r)=\xi^{-1}.\qedhere\]
\end{proof}
Since $\hat{\pi}_r(\cdot,\pm1)>0$ holds $\mathbb{P}$-a.s., there exists a $\phi_r\in L^1(\mathbb{P})$ such that $\phi_r\,\mathrm{d}\mathbb{P}$ is a $\hat{\pi}_r$-invariant probability measure. (See, for example, \cite{alili}.) The pair $(\hat{\pi}_r,\phi_r)$ corresponds to a $\mu_\xi\in M_{1,s}^{\ll}(\Omega\times U)$ with $\mathrm{d}\mu_\xi(\omega,\pm1)=\hat{\pi}_r(\omega,\pm1)\phi_r(\omega)\mathrm{d}\mathbb{P}(\omega)$. By Lemma \ref{Kozlov}, the LLN for the mean velocity of the particle holds under $P_o^{r}$. The limiting velocity is \[\int\sum_{z\in U}\hat{\pi}_r(\omega,z)z\,\phi_r(\omega)\mathrm{d}\mathbb{P}=\xi_{\mu_\xi}\] with the notation in (\ref{ximu}). Therefore, $\xi_{\mu_\xi}^{-1}=E_o^{r}[\tau_1]=\xi^{-1}$ by Lemma \ref{shmeryahu}. In other words, $\mu_\xi\in A_\xi$.

Define $F_r:\Omega\times\{-1,1\}\to\mathbb{R}$ by setting \[F_r(\omega,-1):=\log\zeta(r,T_{-1}\omega)-\lambda(r)\quad\mbox{and}\quad F_r(\omega,1):=-\log\zeta(r,\omega)+\lambda(r).\] Then, recall (\ref{yyedin}) and see that
\begin{equation}
\mathrm{d}\mu_\xi(\omega,z)=\hat{\pi}_r(\omega,z)\phi_r(\omega)\mathrm{d}\mathbb{P}(\omega)=\mathrm{d}(\mu_\xi)^1(\omega)\pi(0,z)\mathrm{e}^{-z\lambda(r)+F_r(\omega,z)+r}\label{veriguut}
\end{equation} for $z\in\{-1,1\}$. In order to conclude that $\mu_\xi$ fits the Ansatz given in Lemma \ref{lagrange}, $F_r\in\mathcal{K}$ needs to be shown. $F_r$ clearly satisfies the mean zero and the closed loop conditions in Definition \ref{K}. For $z\in\{-1,1\}$,
\[\pi(0,z)\mathrm{e}^{-z\lambda(r)+F_r(\omega,z)+r}=\hat{\pi}_r(\omega,z)\leq1\] gives $F_r(\omega,z)\leq|\log\pi(0,z)| +z\lambda(r)-r$. Use the fact that $-F_r(\omega,z)=F_r(T_z\omega,-z)$ to write $$|F_r(\omega,z)|\leq|\log\pi(0,1)|+|\log\pi(1,0)|+|\lambda(r)|-r.$$ The moment condition on $F_r(\cdot,z)$ follows from (A1).

So far, we have obtained a $\mu_\xi$ that fits the Ansatz given in Lemma \ref{lagrange} when $\xi\in(\xi_c,1)$. An analogous construction works for $\xi\in(-1,\bar{\xi}_c)$.

\begin{proof}[Proof of Theorem \ref{explicitformulah} for nearest-neighbor walks on $\mathbb{Z}$]
For any $\xi\in(\xi_c,1)$, the measure $\mu_\xi$ given in (\ref{veriguut}) is the unique minimizer of (\ref{level1ratetilde}) by Lemma \ref{lagrange}. Therefore, $I(\xi) = \mathfrak{I}(\mu_\xi)=r(\xi)-\xi\lambda(r(\xi))$ by (\ref{sifirladik}). Since $\lambda'(r(\xi))=\xi^{-1}$, it is clear that $$I(\xi)=\sup_{r\in\mathbb{R}}\left\{r-\xi\lambda(r)\right\}=\xi\sup_{r\in\mathbb{R}}\left\{r\xi^{-1}-\lambda(r)\right\}=\xi\lambda^{*}(\xi^{-1}).$$

In the proof of Lemma \ref{lifeisrandom} for nearest-neighbor walks on $\mathbb{Z}$, we saw that $r\mapsto\lambda(r)$ is strictly convex and analytic on $(-\infty,r_c)$. By convex duality, $\xi\mapsto I(\xi)$ is strictly convex and analytic on $(\xi_c,1)$.

If $\xi_c=0$, then we have identified $I(\cdot)$ on $(0,1)$. Let us now suppose $\xi_c>0$. Note that $$I'(\xi)=\frac{\mathrm{d}}{\mathrm{d}\xi}[r(\xi)-\xi\lambda(r(\xi))]=r'(\xi)-\lambda(r(\xi))-\xi\lambda'(r(\xi))r'(\xi)=-\lambda(r(\xi)).$$ Therefore, $I(\xi_c)-\xi_cI'(\xi_c+)=r_c$. This implies by convexity that $I(0)\geq r_c$. On the other hand, $$E_o^\omega[\mathrm{e}^{r\tau_1},\tau_1<\infty]=\sum_{k=1}^{\infty}\mathrm{e}^{rk}P_o^\omega(\tau_1=k)\leq\sum_{k=1}^{\infty}\mathrm{e}^{rk}P_o^\omega(X_k=1)\leq\sum_{k=1}^{\infty}\mathrm{e}^{(r-I(0))k + o(k)}<\infty$$ for any $r<I(0)$. Hence, $r_c = I(0)$. The equality $I(\xi_c)-\xi_cI'(\xi_c+)=I(0)$ forces $I(\cdot)$ to be affine linear on $[0,\xi_c]$ with a slope of $I'(\xi_c+)$. In particular, $\xi\mapsto I(\xi)$ is differentiable on $(0,1)$.

Still supposing $\xi_c>0$, fix $\xi\in(0,\xi_c]$. Then, $\frac{\mathrm{d}}{\mathrm{d}r}\left(r-\xi\lambda(r)\right)>0$ for every $r<r_c$. Therefore, $$\sup_{r\in\mathbb{R}}\left\{r-\xi\lambda(r)\right\}=r_c-\xi\lambda(r_c)=I(0)+\xi I'(\xi_c+)=I(\xi).$$ In short, $I(\xi)=\sup_{r\in\mathbb{R}}\left\{r-\xi\lambda(r)\right\}$ for every $\xi\in(0,1)$.

Let us no longer suppose $\xi_c>0$. At $\xi=1$,
\begin{align*}
\sup_{r\in\mathbb{R}}\left\{r-\lambda(r)\right\}&=\sup_{r\in\mathbb{R}}\mathbb{E}\left[r-\log E_o^\omega\left[\mathrm{e}^{r\tau_1},\tau_1<\infty\right]\right]=\lim_{r\to-\infty}\mathbb{E}\left[-\log E_o^\omega\left[\mathrm{e}^{r(\tau_1-1)},\tau_1<\infty\right]\right]\\
&=\mathbb{E}\left[-\log P_o^\omega\left(\tau_1=1\right)\right]=\mathbb{E}\left[-\log\pi(0,1)\right]=-\lim_{n\to\infty}\frac{1}{n}\log P_o^\omega\left(X_n=n\right)=I(1).
\end{align*}
It is easy to check that $I(\xi)=\sup_{r\in\mathbb{R}}\left\{r-\xi\lambda(r)\right\}=\infty$ when $\xi>1$. This concludes the proof of Theorem \ref{explicitformulah} for $\xi\geq0$. The arguments regarding $\xi<0$ are similar.
\end{proof}

\section{Walks with bounded jumps on $\mathbb{Z}$}\label{dolapdere}

Recall the statement of Lemma \ref{lifeisrandom}. We start this section by constructing a new (tilted) environment kernel $\hat{\pi}_r$ for every $r<r_c$. We then prove that $r\mapsto\lambda(r)$ exists and that it is differentiable on $(-\infty,r_c)$. At that point, we note that if there exists a $\phi_r\in L^1(\mathbb{P})$ such that $\phi_r\,\mathrm{d}\mathbb{P}$ is a $\hat{\pi}_r$-invariant probability measure, then $\mu_\xi\in M_1(\Omega\times\mathcal{R})$ with $\mathrm{d}\mu_\xi(\omega,z)=\hat{\pi}_r(\omega,z)\phi_r(\omega)\mathrm{d}\mathbb{P}(\omega)$ fits the Ansatz given in Lemma \ref{lagrange} for $\xi=(\lambda'(r))^{-1}$. We proceed by constructing such a $\phi_r$. Finally, we prove Lemma \ref{lifeisrandom}, Theorem \ref{explicitformulah} and Proposition \ref{Wronskian}.

\begin{remark}
Some of the notation (e.g., $\zeta(r,\omega), r_c, \hat{\pi}_r, \phi_r$ and $F_r$) introduced in Section \ref{vandiseksin} is reintroduced in Section \ref{dolapdere} in a slightly different way. This is done in order to emphasize the fact that the arguments in these two sections are parallel. Note that this practice does not cause any confusion since Sections \ref{vandiseksin} and \ref{dolapdere} can be read independently of each other.
\end{remark}

Many of the arguments in this section use the following lemma.

\begin{lemma}\label{muhacir}
Given $m\in\mathbb{Z}$ and $\epsilon>0$, suppose there exist two functions $L:\left((-\infty,m)\cap\mathbb{Z}\right)\times\mathcal{R}\to[0,1]$ and $v:\mathbb{Z}\to\mathbb{R}$ such that $L(y,\pm1)\geq\epsilon$, $\sum_{z\in\mathcal{R}}L(y,z)=1$ and $v(y)=\sum_{z\in\mathcal{R}}L(y,z)v(y+z)$ for any $y<m$. The function $L$ defines a Markov chain and, for any $x<m$, induces a probability measure $Q_x$ on paths starting at $x$. $E_x^Q$ denotes expectation under $Q_x$.

If $Q_x(\tau_m<\infty)=1$ and $x'<x$, then
\begin{equation}\label{faikinbaci}
|v(x)-v(x')|\leq \left(1-\epsilon^B\right)^{\frac{m-x}{B}}\sup_{0\leq z<B \atop 0\leq z'<B}\left[v(m+z)-v(m+z')\right].
\end{equation}
\end{lemma}

\begin{proof}
Fix $x'<x<m$. For any $k\geq0$ with $x+(k+1)B\leq m$, $$v(x')=E_{x'}^Q\left[v\left(X_{\tau_{x+kB}}\right)\right]=\sum_{z=0}^{B-1}Q_{x'}\left(X_{\tau_{x+kB}}=x+kB+z\right)v(x+kB+z).$$ There exists an $x_k\in\mathbb{Z}$ such that $x+kB\leq x_k<x+(k+1)B$ and $v(x_k)\leq v(x')$. The collection of $x_k$'s constitute a set $S:=\left\{x_k:0\leq k\leq\frac{m-x}{B}-1\right\}$. Let $\tau_S:=\inf\left\{k\geq0:X_k\in S\right\}$. Observe that
\begin{align*}
v(x)=E_x^Q\left[v\left(X_{\tau_S\wedge\tau_m}\right)\right]&=E_x^Q\left[v\left(X_{\tau_S}\right),\tau_S<\infty\right]+E_x^Q\left[v\left(X_{\tau_m}\right),\tau_S=\infty\right]\\
&\leq Q_x(\tau_S<\infty)v(x')\quad\ \ \,+Q_x(\tau_S=\infty)\sup_{0\leq z<B}v(m+z)\\
&=v(x')+Q_x(\tau_S=\infty)\left(\sup_{0\leq z<B}v(m+z)-v(x')\right).
\end{align*}
On the other hand, $v(x')=E_{x'}^Q\left[v\left(X_{\tau_m}\right)\right]\geq\inf_{0\leq z'<B}v(m+z')$. Therefore, $$v(x)-v(x')\leq Q_x(\tau_S=\infty)\sup_{0\leq z<B \atop 0\leq z'<B}\left[v(m+z)-v(m+z')\right].$$
It is easy to see that $Q_x(\tau_S=\infty)\leq\left(1-\epsilon^B\right)^{\frac{m-x}{B}}$. This proves half of (\ref{faikinbaci}). The other half is proved similarly.
\end{proof}

\subsection{Construction of a new environment kernel $\hat{\pi}_r$}\label{haifaguzel}

Recall our assumptions:
\begin{enumerate}
\item [(A1)] There exists an $\alpha>0$ such that $\int|\log\pi(0,z)|^{1+\alpha}\mathrm{d}\mathbb{P}<\infty$ for each $z\in\mathcal{R}$.
\item [(A2)] There exists a $\delta>0$ such that $\mathbb{P}(\pi(0,\pm 1)\geq\delta)=1$.
\end{enumerate}

Let $\zeta(r,\omega):=E_o^\omega\left[\mathrm{e}^{r\tau_1},\tau_1<\infty\right]$ for any $r\in\mathbb{R}$. If $\zeta(r,\omega)<\infty$, then
\begin{align*}
\zeta(r,\omega)&=\sum_{z\in\mathcal{R}}\pi(0,z)\mathrm{e}^rE_z^\omega\left[\mathrm{e}^{r\tau_1},\tau_1<\infty\right]\geq\pi(0,-1)\mathrm{e}^rE_{-1}^\omega\left[\mathrm{e}^{r\tau_1},\tau_1<\infty\right]\\
&\geq\delta\mathrm{e}^r\left(E_{-1}^\omega\left[\mathrm{e}^{r\tau_1},\tau_1<\infty, X_{\tau_o}\geq1\right] + E_{-1}^\omega\left[\mathrm{e}^{r\tau_1},\tau_1<\infty, X_{\tau_o}=0\right]\right)\\
&=\delta\mathrm{e}^r\left(E_{-1}^\omega\left[\mathrm{e}^{r\tau_o},\tau_o<\infty, X_{\tau_o}\geq1\right] + E_{-1}^\omega\left[\mathrm{e}^{r\tau_o},\tau_o<\infty, X_{\tau_o}=0\right]E_o^\omega\left[\mathrm{e}^{r\tau_1},\tau_1<\infty\right]\right)\\
&\geq\min(1,\zeta(r,\omega))\delta\mathrm{e}^r\zeta(r,T_{-1}\omega).
\end{align*}
Therefore, $\{\omega:\zeta(r,\omega)<\infty\}$ is $T$-invariant, and its $\mathbb{P}$-probability is $0$ or $1$. The function $r\mapsto\zeta(r,\omega)$ is strictly increasing. There exists an $r_c\geq0$ such that $\mathbb{P}\left(\omega: \zeta(r,\omega)<\infty\right)=1$ if $r<r_c$ and $\mathbb{P}\left(\omega: \zeta(r,\omega)=\infty\right)=1$ if $r>r_c$. When $r<r_c$, $$E_o^\omega\left[\mathrm{e}^{r\tau_1},\tau_1<\infty\right]\geq E_o^\omega\left[\mathrm{e}^{r\tau_1}, X_1=-1, X_2=0, \tau_1<\infty\right]\geq\left(\delta\mathrm{e}^r\right)^2E_o^\omega\left[\mathrm{e}^{r\tau_1},\tau_1<\infty\right].$$ This shows that $\delta\mathrm{e}^r\leq1$ and $r_c\leq-\log\delta<\infty$. For $r<r_c$ and $n\geq2$,
\begin{align*}
E_o^\omega\left[\mathrm{e}^{r\tau_n},\tau_n<\infty\right]&=\sum_{z=1}^{B}E_o^\omega\left[\mathrm{e}^{r\tau_1},\tau_1<\infty,X_{\tau_1}=z\right]E_z^\omega\left[\mathrm{e}^{r\tau_n},\tau_n<\infty\right]\\
&=\sum_{z=1}^{B}E_o^\omega\left[\mathrm{e}^{r\tau_1},\tau_1<\infty,X_{\tau_1}=z\right]E_o^{T_z\omega}\left[\mathrm{e}^{r\tau_{n-z}},\tau_{n-z}<\infty\right].
\end{align*} By induction, $\mathbb{P}\left(\omega: E_o^\omega\left[\mathrm{e}^{r\tau_n},\tau_n<\infty\right]<\infty\right)=1$.

From now on, consider $r<r_c$. For $x<n$, note that
\begin{align}
u_{r,n}(\omega,x):\!&=\frac{E_x^\omega\left[\mathrm{e}^{r\tau_n},\tau_n<\infty\right]}{E_o^\omega\left[\mathrm{e}^{r\tau_n},\tau_n<\infty\right]}=\sum_{z\in\mathcal{R}}\pi(x,x+z)\mathrm{e}^r\frac{E_{x+z}^\omega\left[\mathrm{e}^{r\tau_n},\tau_n<\infty\right]}{E_o^\omega\left[\mathrm{e}^{r\tau_n},\tau_n<\infty\right]}\nonumber\\
&=\sum_{z\in\mathcal{R}}\pi(x,x+z)\mathrm{e}^ru_{r,n}(\omega,x+z).\nonumber\\
1&=\sum_{z\in\mathcal{R}}\pi(x,x+z)\mathrm{e}^r\frac{u_{r,n}(\omega,x+z)}{u_{r,n}(\omega,x)}=:\sum_{z\in\mathcal{R}}\hat{\pi}_{r,n}(x,x+z)\label{huseyindayi}
\end{align} defines a new (random) transition kernel $\hat{\pi}_{r,n}(x,x+z)$ for $x<n$. It is clear that the jumps under $\hat{\pi}_{r,n}$ are bounded by $B$. If $x<y<n$, then
$$E_{y}^\omega\left[\mathrm{e}^{r\tau_n},\tau_n<\infty\right]\geq E_{y}^\omega\left[\mathrm{e}^{r\tau_n},X_1=y-1,\ldots,X_{y-x}=x,\tau_n<\infty\right]\geq(\delta\mathrm{e}^r)^{y-x}E_{x}^\omega\left[\mathrm{e}^{r\tau_n},\tau_n<\infty\right].$$ Similarly, $E_{x}^\omega\left[\mathrm{e}^{r\tau_n},\tau_n<\infty\right]\geq(\delta\mathrm{e}^r)^{y-x}E_{y}^\omega\left[\mathrm{e}^{r\tau_n},\tau_n<\infty\right]$. Therefore,
\begin{equation}
(\delta\mathrm{e}^r)^{|y-x|}\leq\frac{u_{r,n}(\omega,y)}{u_{r,n}(\omega,x)}\leq(\delta\mathrm{e}^r)^{-|y-x|}.\label{isilam}
\end{equation}
Putting (\ref{huseyindayi}) and (\ref{isilam}) together, we obtain the following ellipticity bound:
\begin{equation}
\mathbb{P}\left(\hat{\pi}_{r,n}(x,x\pm1)\geq(\delta\mathrm{e}^r)^2\right)=1\mbox{ for every }x<n-1.\label{reginbogin}
\end{equation}\pagebreak

\begin{lemma}\label{Qgot}
If $0<x<n+B$, then $(\delta\mathrm{e}^r)^{4(B-1)}\leq u_{r,n}(\omega,x)E_o^\omega\left[\mathrm{e}^{r\tau_x},\tau_x<\infty\right]\leq(\delta\mathrm{e}^r)^{-4(B-1)}$ for $\mathbb{P}$-a.e.\ $\omega$.
\end{lemma}
\begin{proof}
Suppose $0<x\leq n-B$. Observe that
\begin{align}
u_{r,n}(\omega,x)^{-1}&=\frac{E_o^\omega\left[\mathrm{e}^{r\tau_n},\tau_n<\infty\right]}{E_x^\omega\left[\mathrm{e}^{r\tau_n},\tau_n<\infty\right]}=\frac{1}{E_x^\omega\left[\mathrm{e}^{r\tau_n},\tau_n<\infty\right]}\sum_{z=0}^{B-1}E_o^\omega\left[\mathrm{e}^{r\tau_n},\tau_n<\infty,X_{\tau_x}=x+z\right]\nonumber\\
&=\sum_{z=0}^{B-1}E_o^\omega\left[\mathrm{e}^{r\tau_x},\tau_x<\infty,X_{\tau_x}=x+z\right]\frac{E_{x+z}^\omega\left[\mathrm{e}^{r\tau_n},\tau_n<\infty\right]}{E_x^\omega\left[\mathrm{e}^{r\tau_n},\tau_n<\infty\right]}\nonumber\\
&=\sum_{z=0}^{B-1}E_o^\omega\left[\mathrm{e}^{r\tau_x},\tau_x<\infty,X_{\tau_x}=x+z\right]\frac{u_{r,n}(\omega,x+z)}{u_{r,n}(\omega,x)}.\label{ucuzkurtuldum}
\end{align}
It follows immediately from (\ref{isilam}) that
\begin{equation}\label{deyyus}
(\delta\mathrm{e}^r)^{(B-1)}\leq u_{r,n}(\omega,x)E_o^\omega\left[\mathrm{e}^{r\tau_x},\tau_x<\infty\right]\leq(\delta\mathrm{e}^r)^{-(B-1)}.
\end{equation}

Next, suppose $n-B<x<n$. Note that (\ref{ucuzkurtuldum}) still holds. If $x+z<n$, then
\begin{equation}\label{hizlanmakiyi}
(\delta\mathrm{e}^r)^{(B-1)}\leq\frac{u_{r,n}(\omega,x+z)}{u_{r,n}(\omega,x)}\leq(\delta\mathrm{e}^r)^{-(B-1)}
\end{equation} again by (\ref{isilam}). On the other hand, if $x+z\geq n$, then 
\begin{equation}\label{havvaanan}
\frac{u_{r,n}(\omega,x+z)}{u_{r,n}(\omega,x)}=\frac{E_{x+z}^\omega\left[\mathrm{e}^{r\tau_n},\tau_n<\infty\right]}{E_x^\omega\left[\mathrm{e}^{r\tau_n},\tau_n<\infty\right]}=\frac{1}{E_x^\omega\left[\mathrm{e}^{r\tau_n},\tau_n<\infty\right]}.
\end{equation} However, for any $m\geq n+B$, 
\begin{equation}\label{ruyalarda}
(\delta\mathrm{e}^r)^{2(B-1)}\leq(\delta\mathrm{e}^r)^{(B-1)}\frac{u_{r,m}(\omega,n)}{u_{r,m}(\omega,x)}\leq\frac{1}{E_x^\omega\left[\mathrm{e}^{r\tau_n},\tau_n<\infty\right]}\leq(\delta\mathrm{e}^r)^{-(B-1)}\frac{u_{r,m}(\omega,n)}{u_{r,m}(\omega,x)}\leq(\delta\mathrm{e}^r)^{-2(B-1)}.
\end{equation} In (\ref{ruyalarda}), the inner inequalities follow from (\ref{deyyus}) after an appropriate change of variables, and the outer inequalities hold by (\ref{isilam}). Use (\ref{ucuzkurtuldum}) in combination with (\ref{hizlanmakiyi}), (\ref{havvaanan}) and (\ref{ruyalarda}) to deduce that
\begin{equation}\label{deyyusiki}
(\delta\mathrm{e}^r)^{2(B-1)}\leq u_{r,n}(\omega,x)E_o^\omega\left[\mathrm{e}^{r\tau_x},\tau_x<\infty\right]\leq(\delta\mathrm{e}^r)^{-2(B-1)}.
\end{equation}

If $x=n$, there is nothing to prove. Finally, suppose $n<x<n+B$. It is easy to see that
\begin{equation}\label{goncayla}
u_{r,n}(\omega,x)E_o^\omega\left[\mathrm{e}^{r\tau_x},\tau_x<\infty\right]=\frac{E_o^\omega\left[\mathrm{e}^{r\tau_x},\tau_x<\infty\right]}{E_o^\omega\left[\mathrm{e}^{r\tau_n},\tau_n<\infty\right]}=\frac{1}{u_{r,x}(\omega,n)E_o^\omega\left[\mathrm{e}^{r\tau_n},\tau_n<\infty\right]}\cdot E_n^\omega\left[\mathrm{e}^{r\tau_x},\tau_x<\infty\right].
\end{equation} Reversing the roles of $x$ and $n$ in both (\ref{ruyalarda}) and (\ref{deyyusiki}) gives upper and lower bounds for the terms on the RHS of (\ref{goncayla}). This implies the desired result.
\end{proof}

In order to indicate the $\omega$-dependence of $\hat{\pi}_{r,n}$, the probability measure it induces on paths starting at any $x<n$ is denoted by $Q_x^{n,\omega}$.

\begin{lemma}\label{ewayak}
For every $x<n$ and $\mathbb{P}$-a.e.\ $\omega$, $Q_x^{n,\omega}(\tau_n<\infty)=1$.
\end{lemma}
\begin{proof}
For any path $(x_j)_{0\leq j\leq k}$ with $x_o=x$, $x_j<n$ and $x_{j+1}-x_j\in\mathcal{R}$, it follows from (\ref{huseyindayi}) that $$Q_x^{n,\omega}(X_1 = x_1,\ldots, X_k = x_k)=P_x^\omega(X_1 = x_1,\ldots, X_k = x_k)\mathrm{e}^{rk}\frac{u_{r,n}(\omega,x_k)}{u_{r,n}(\omega,x)}.$$
Also, note that $P_x^\omega\left(\left.X_{\tau_n}\geq n\,\right|\tau_n<\infty\right)=1$. Therefore,
$$Q_x^{n,\omega}(\tau_n<\infty)=E_x^\omega[\mathrm{e}^{r\tau_n}\frac{u_{r,n}(\omega,X_{\tau_n})}{u_{r,n}(\omega,x)},\tau_n<\infty]=E_x^\omega[\mathrm{e}^{r\tau_n}\frac{E_{X_{\tau_n}}^\omega\left[\mathrm{e}^{r\tau_n},\tau_n<\infty\right]}{E_x^\omega\left[\mathrm{e}^{r\tau_n},\tau_n<\infty\right]},\tau_n<\infty] = 1.\qedhere$$
\end{proof}

\begin{lemma}
For every $x,z\in\mathbb{Z}$ and $\mathbb{P}$-a.e.\ $\omega$, $u_r(\omega,x):=\lim_{n\to\infty}u_{r,n}(\omega,x)$ exists and
\begin{equation}\label{cangorecek}
(\delta\mathrm{e}^r)^{|z|}\leq u_r(T_x\omega,z)=\frac{u_r(\omega,x+z)}{u_r(\omega,x)}\leq(\delta\mathrm{e}^r)^{-|z|}.
\end{equation}
\end{lemma}
\begin{proof}
Given $x\in\mathbb{Z}$, take any $n_1, n_2\in\mathbb{N}$ such that $x<n_1<n_2$. For every $y\in\mathbb{Z}$ with $n_1\leq y<n_1+B$, Lemma \ref{Qgot} implies that $$(\delta\mathrm{e}^r)^{8(B-1)}\leq\frac{u_{r,n_2}(\omega,y)}{u_{r,n_1}(\omega,y)}\leq(\delta\mathrm{e}^r)^{-8(B-1)}.$$
Since $\hat{\pi}_{r,n_1}$ is defined in (\ref{huseyindayi}) via a Doob $h$-transform, it is not surprising that
\begin{align*}
\sum_{z\in\mathcal{R}}\hat{\pi}_{r,n_1}(x,x+z)\frac{u_{r,n_2}(\omega,x+z)}{u_{r,n_1}(\omega,x+z)}&=\sum_{z\in\mathcal{R}}\pi(x,x+z)\mathrm{e}^r\frac{u_{r,n_1}(\omega,x+z)}{u_{r,n_1}(\omega,x)}\frac{u_{r,n_2}(\omega,x+z)}{u_{r,n_1}(\omega,x+z)}\\&=\sum_{z\in\mathcal{R}}\pi(x,x+z)\mathrm{e}^r\frac{u_{r,n_2}(\omega,x+z)}{u_{r,n_1}(\omega,x)}\\&=\frac{u_{r,n_2}(\omega,x)}{u_{r,n_1}(\omega,x)}.
\end{align*}
Therefore, Lemma \ref{muhacir} implies that
$$\left|\frac{u_{r,n_2}(\omega,x)}{u_{r,n_1}(\omega,x)}-1\right| = \left|\frac{u_{r,n_2}(\omega,x)}{u_{r,n_1}(\omega,x)}-\frac{u_{r,n_2}(\omega,0)}{u_{r,n_1}(\omega,0)}\right|\leq c(r)^{n_1-|x|}\left(\delta\mathrm{e}^r\right)^{-8(B-1)}$$
where $c(r) := \left(1-\left(\delta\mathrm{e}^r\right)^{2B}\right)^{1/B}<1$. Substitute $y=0$ in (\ref{isilam}) and conclude that
$$\left|u_{r,n_2}(\omega,x)-u_{r,n_1}(\omega,x)\right|=u_{r,n_1}(\omega,x)\left|\frac{u_{r,n_2}(\omega,x)}{u_{r,n_1}(\omega,x)}-1\right|\leq c(r)^{n_1-|x|}\left(\delta\mathrm{e}^r\right)^{-8(B-1)-|x|}.$$ In particular, $\left(u_{r,n}(\omega,x)\right)_{n>x}$ is a Cauchy sequence. Therefore, $u_r(\omega,x):=\lim_{n\to\infty}u_{r,n}(\omega,x)$ exists.

For every $x,z\in\mathbb{Z}$ and $\mathbb{P}$-a.e.\ $\omega$,
\begin{align}
\frac{u_r(\omega,x+z)}{u_r(\omega,x)}&=\lim_{n\to\infty}\!\!\frac{u_{r,n}(\omega,x+z)}{u_{r,n}(\omega,x)}=\lim_{n\to\infty}\!\!\frac{E_{x+z}^\omega\left[\mathrm{e}^{r\tau_n},\tau_n<\infty\right]}{E_x^\omega\left[\mathrm{e}^{r\tau_n},\tau_n<\infty\right]}=\lim_{n\to\infty}\!\!\frac{E_z^{T_x\omega}\left[\mathrm{e}^{r\tau_{n-x}},\tau_{n-x}<\infty\right]}{E_o^{T_x\omega}\left[\mathrm{e}^{r\tau_{n-x}},\tau_{n-x}<\infty\right]}\label{feyziogluyla}\\&=u_r(T_x\omega,z).\nonumber
\end{align}
Finally, note that the inequalities in (\ref{cangorecek}) follow from (\ref{isilam}).
\end{proof}

\begin{definition}
For every $z\in\mathcal{R}$ and $\mathbb{P}$-a.e.\ $\omega$, let 
\begin{equation}\label{birincicinko}
\hat{\pi}_r(\omega,z):=\pi(0,z)\mathrm{e}^r u_r(\omega,z).
\end{equation}
\end{definition}

It follows immediately from (\ref{huseyindayi}) that $\hat{\pi}_r:\Omega\times\mathcal{R}\to[0,1]$ is an environment kernel in the sense of Definition \ref{ortamkeli}. For every $x\in\mathbb{Z}$, in order to simplify the notation, $P_x^{\hat{\pi}_r,\omega}, E_x^{\hat{\pi}_r,\omega}, P_x^{\hat{\pi}_r}$ and $E_x^{\hat{\pi}_r}$ are denoted by $P_x^{r,\omega}, E_x^{r,\omega}, P_x^r$ and $E_x^r$, respectively.
 
\begin{lemma} For every $z\in\mathcal{R}$ and $\mathbb{P}$-a.e.\ $\omega$, 
\begin{equation}\label{ikincicinko}
\hat{\pi}_r(\omega,z)\geq\left(\delta\mathrm{e}^r\right)^{|z|}\mathrm{e}^r\pi(0,z).
\end{equation} In particular, $\hat{\pi}_r$ satisfies the following ellipticity condition: 
\begin{equation}\label{tombala}
\mathbb{P}\left(\omega: \hat{\pi}_r(\omega,\pm1)\geq\left(\delta\mathrm{e}^r\right)^2\right)=1.
\end{equation}
\end{lemma}
 
\begin{proof}
(\ref{cangorecek}) and (\ref{birincicinko}) imply (\ref{ikincicinko}) which gives (\ref{tombala}) since (A2) holds.
\end{proof}
 
\begin{lemma}\label{kayahan}
For every $n\geq1$, $P_o^r(\tau_n<\infty)=1$.
\end{lemma}

\begin{proof}
Recall the proof of Lemma \ref{ewayak}. For every $n\geq1$ and $\mathbb{P}$-a.e.\ $\omega$,
\begin{align*}
P_o^{r,\omega}\left(\tau_n<\infty\right)&=E_o^\omega\left[\mathrm{e}^{r\tau_n}u_r(\omega,X_{\tau_n}), \tau_n<\infty\right]\\
&=\sum_{z=0}^{B-1}E_o^\omega\left[\mathrm{e}^{r\tau_n}, X_{\tau_n}=n+z, \tau_n<\infty\right]u_r(\omega,n+z)\\
&=\lim_{m\to\infty}\sum_{z=0}^{B-1}E_o^\omega\left[\mathrm{e}^{r\tau_n}, X_{\tau_n}=n+z, \tau_n<\infty\right]u_{r,m}(\omega,n+z)\\
&=\lim_{m\to\infty}E_o^\omega\left[\mathrm{e}^{r\tau_n}u_{r,m}(\omega,X_{\tau_n}), \tau_n<\infty\right]\\
&=\lim_{m\to\infty}\frac{E_o^\omega\left[\mathrm{e}^{r\tau_n}E_{X_{\tau_n}}^\omega\left[\mathrm{e}^{r\tau_m}, \tau_m<\infty\right], \tau_n<\infty\right]}{E_o^\omega\left[\mathrm{e}^{r\tau_m}, \tau_m<\infty\right]}\\
&=\lim_{m\to\infty}\frac{E_o^\omega\left[\mathrm{e}^{r\tau_m}, \tau_m<\infty\right]}{E_o^\omega\left[\mathrm{e}^{r\tau_m}, \tau_m<\infty\right]}=1.\qedhere
\end{align*}
\end{proof}

\begin{lemma}\label{mavisap}
For every $m\geq1$ and $\mathbb{P}$-a.e.\ $\omega$,
\begin{equation}\label{boundoyleolmazboyleolur}
E_o^{r,\omega}\left[\tau_1^m\right]\leq\frac{m!}{(r_c-r)^m}(\delta\mathrm{e}^r)^{-2B}=:H_m(r).
\end{equation}
\end{lemma}

\begin{proof}
It follows from (\ref{isilam}) and (\ref{deyyus}) that $E_o^\omega\left[\mathrm{e}^{r\tau_1},\tau_1<\infty\right]\leq(\delta\mathrm{e}^r)^{-B}$ for $\mathbb{P}$-a.e.\ $\omega$. By the monotone convergence theorem, this bound holds for $r=r_c$ as well. Note that
\begin{align*}
E_o^{r,\omega}\left[\mathrm{e}^{(r_c-r)\tau_1}\right]&=E_o^{r,\omega}\left[\mathrm{e}^{(r_c-r)\tau_1},\tau_1<\infty\right]=E_o^\omega\left[\mathrm{e}^{r_c\tau_1}u_r(\omega,X_{\tau_1}),\tau_1<\infty\right]\\
&\leq E_o^\omega\left[\mathrm{e}^{r_c\tau_1}(\delta\mathrm{e}^r)^{-B},\tau_1<\infty\right]\leq (\delta\mathrm{e}^r)^{-B}(\delta\mathrm{e}^{r_c})^{-B}\leq(\delta\mathrm{e}^r)^{-2B}.
\end{align*} Here, Lemma \ref{kayahan} and (\ref{cangorecek}) imply the first equality and the first inequality, respectively. For every $m\geq1$ and $a\in\mathbb{R}^+$, $\mathrm{e}^{a}=\sum_{n=0}^\infty\frac{a^n}{n!}\geq\frac{a^m}{m!}$. Therefore,
$$\frac{(r_c-r)^m}{m!}E_o^{r,\omega}\left[\tau_1^m\right]\leq E_o^{r,\omega}\left[\mathrm{e}^{(r_c-r)\tau_1}\right]\leq(\delta\mathrm{e}^r)^{-2B}.\qedhere$$
\end{proof}

\begin{lemma}\label{bugeceler}
For $\mathbb{P}$-a.e.\ $\omega$,
$$\lim_{n\to\infty}\frac{1}{n}E_o^{r,\omega}\left[\tau_n\right]=\mathbb{E}\left[\lim_{x\to-\infty}P_x^{r,\omega}\left(X_{\tau_o}=0\right)E_o^{r,\omega}\left[\tau_1\right]\right]=:g(r).$$
\end{lemma}

\begin{proof}
Let $c(r) := \left(1-\left(\delta\mathrm{e}^r\right)^{2B}\right)^{1/B}<1$. For every $n\geq1$ and $\mathbb{P}$-a.e.\ $\omega$,
\begin{align}
E_o^{r,\omega}\left[\tau_n\right] &= \sum_{i=1}^nE_o^{r,\omega}\left[\tau_i - \tau_{i-1}\right] = \sum_{i=1}^nE_o^{r,\omega}\left[\tau_i - \tau_{i-1}, X_{\tau_{i-1}}=i-1\right]\nonumber\\
& = \sum_{i=1}^nP_o^{r,\omega}\left(X_{\tau_{i-1}}=i-1\right)E_{i-1}^{r,\omega}\left[\tau_i\right]\nonumber\\
& \leq \sum_{i=1}^n\left(\lim_{x\to-\infty}P_x^{r,\omega}\left(X_{\tau_{i-1}}=i-1\right) + c(r)^{i-1}\right)E_{i-1}^{r,\omega}\left[\tau_i\right]\label{bengurion}\\
& = \sum_{i=1}^n\left(\lim_{x\to-\infty}P_x^{r,T_{i-1}\omega}\left(X_{\tau_o}=0\right) + c(r)^{i-1}\right)E_o^{r,T_{i-1}\omega}\left[\tau_1\right]\nonumber\\
& \leq \frac{H_1(r)}{1-c(r)} + \sum_{i=1}^n\lim_{x\to-\infty}P_x^{r,T_{i-1}\omega}\left(X_{\tau_o}=0\right)E_o^{r,T_{i-1}\omega}\left[\tau_1\right]\nonumber
\end{align} where (\ref{bengurion}) and the existence of $\lim_{x\to-\infty}P_x^{r,\omega}\left(X_{\tau_{i-1}}=i-1\right)$ follow from Lemma \ref{muhacir}. Therefore,
$$\limsup_{n\to\infty}\frac{1}{n}E_o^{r,\omega}\left[\tau_n\right]\leq\mathbb{E}\left[\lim_{x\to-\infty}P_x^{r,\omega}\left(X_{\tau_o}=0\right)E_o^{r,\omega}\left[\tau_1\right]\right]$$ by the ergodic theorem. The proof of the other direction is similar.
\end{proof}

\subsection{Differentiability of $r\mapsto\lambda(r)$}\label{otayak}

For every $r<r_c$, $n\geq1$ and $\mathbb{P}$-a.e.\ $\omega$, let $$\lambda_n(r,\omega):=\log E_o^\omega\left[\mathrm{e}^{r\tau_n},\tau_n<\infty\right].$$

\begin{lemma}\label{ongbak}
For every $r<r_c$ and $\mathbb{P}$-a.e.\ $\omega$, $$\lambda(r):=\lim_{n\to\infty}\frac{1}{n}\lambda_n(r,\omega)=-\mathbb{E}\left[\log u_r(\cdot,1)\right].$$
\end{lemma}
\begin{proof}
It follows from Lemma \ref{Qgot} that $(\delta\mathrm{e}^r)^{4(B-1)}\leq u_r(\omega,n)E_o^\omega\left[\mathrm{e}^{r\tau_n},\tau_n<\infty\right]\leq(\delta\mathrm{e}^r)^{-4(B-1)}$ for every $n\geq1$. Therefore, $$\lim_{n\to\infty}\left(\frac{1}{n}\log u_r(\omega,n) + \frac{1}{n}\log E_o^\omega\left[\mathrm{e}^{r\tau_n},\tau_n<\infty\right]\right) = 0.$$ However, by (\ref{feyziogluyla}), $$u_r(\omega,n)=\prod_{i=0}^{n-1}\frac{u_r(\omega,i+1)}{u_r(\omega,i)} = \prod_{i=0}^{n-1}u_r(T_i\omega,1).$$ Hence, it follows from the ergodic theorem that $$\lim_{n\to\infty}\frac{1}{n}\log E_o^\omega\left[\mathrm{e}^{r\tau_n},\tau_n<\infty\right] = -\lim_{n\to\infty}\frac{1}{n}\log u_r(\omega,n) = -\lim_{n\to\infty}\frac{1}{n}\sum_{i=0}^{n-1}\log u_r(T_i\omega,1) = -\mathbb{E}\left[\log u_r(\cdot,1)\right].\qedhere$$ 
\end{proof}

In this subsection, we prove that $r\mapsto\lambda(r)$ is differentiable on $(-\infty,r_c)$. For that purpose, we first obtain certain bounds on $\lambda_n'(r,\omega)$ and $\lambda_n''(r,\omega)$. These bounds are given in the next two lemmas which involve the function $$G(\omega):=\inf_{1\leq z'\leq B}\pi\left(-1-z',-1\right)\inf_{0\leq z<B}\pi(-1,z).$$
For every $n\geq1$, $0\leq z<B$ and $\mathbb{P}$-a.e.\ $\omega$, note that
\begin{align}
P_o^{r,\omega}\left(X_{\tau_n}=n+z\right)&\geq P_o^{r,\omega}\left(X_{\tau_{n-1}}=n-1, X_{\tau_n}=n+z\right)\nonumber\\
&\geq\inf_{1\leq z'\leq B}\hat{\pi}_r\left(T_{n-1-z'}\omega,z'\right)\hat{\pi}_r(T_{n-1}\omega,z+1)\nonumber\\
&\geq(\delta\mathrm{e}^r)^{2B}\mathrm{e}^{2r}\inf_{1\leq z'\leq B}\pi\left(n-1-z',n-1\right)\pi(n-1,n+z)\label{madakdamak}\\
&\geq(\delta\mathrm{e}^r)^{2B}\mathrm{e}^{2r}G(T_n\omega)\nonumber
\end{align} where (\ref{madakdamak}) follows from (\ref{ikincicinko}).

\begin{lemma}\label{tornacihuso}
For every $r<r_c$, $n\geq1$ and $\mathbb{P}$-a.e.\ $\omega$,
\begin{equation}\label{cirkinsonja}
\left|\lambda_n'(r,\omega)-E_o^{r,\omega}\left[\tau_n\right]\right|\leq\frac{W_1(r)}{G(T_n\omega)}
\end{equation} is satisfied with $$W_1(r):=\frac{(\delta\mathrm{e}^r)^{-2B}\mathrm{e}^{-2r}H_1(r)c(r)^{1-B}}{1-c(r)},\quad H_1(r)\mbox{ as in (\ref{boundoyleolmazboyleolur}), and}\quad c(r) := \left(1-\left(\delta\mathrm{e}^r\right)^{2B}\right)^{1/B}<1.$$
\end{lemma}

\begin{proof}
For every $r<r_c$, $n\geq1$ and $\mathbb{P}$-a.e.\ $\omega$,
\begin{align*}
\lambda_n'(r,\omega)&=\frac{E_o^\omega\left[\tau_n\mathrm{e}^{r\tau_n},\tau_n<\infty\right]}{E_o^\omega\left[\mathrm{e}^{r\tau_n},\tau_n<\infty\right]}=\frac{E_o^{r,\omega}\left[\tau_nu_r\left(\omega,X_{\tau_n}\right)^{-1}\right]}{E_o^{r,\omega}\left[u_r\left(\omega,X_{\tau_n}\right)^{-1}\right]}=\frac{\sum_{z=0}^{B-1}E_o^{r,\omega}\left[\tau_n, X_{\tau_n}=n+z\right]u_r\left(\omega,n+z\right)^{-1}}{\sum_{z=0}^{B-1}P_o^{r,\omega}\left(X_{\tau_n}=n+z\right)u_r\left(\omega,n+z\right)^{-1}}.
\end{align*} Therefore,
\begin{equation}\label{turgayevlendi}
\inf_{0\leq z<B}E_o^{r,\omega}\left[\left.\tau_n\,\right| X_{\tau_n}=n+z\right]\leq\lambda_n'(r,\omega)\leq\sup_{0\leq z<B}E_o^{r,\omega}\left[\left.\tau_n\,\right| X_{\tau_n}=n+z\right].
\end{equation}
If $1\leq i\leq n$ and $0\leq z<B$, then
\begin{align}
E_o^{r,\omega}\left[\tau_i - \tau_{i-1}, X_{\tau_n}=n+z\right]&=\sum_{z'=0}^{B-1}E_o^{r,\omega}\left[\tau_i - \tau_{i-1}, X_{\tau_i}=i+z', X_{\tau_n}=n+z\right]\nonumber\\
&=\sum_{z'=0}^{B-1}E_o^{r,\omega}\left[\tau_i - \tau_{i-1}, X_{\tau_i}=i+z'\right]P_{i+z'}^{r,\omega}\left(X_{\tau_n}=n+z\right)\nonumber\\
&\leq\sum_{z'=0}^{B-1}E_o^{r,\omega}\left[\tau_i - \tau_{i-1}, X_{\tau_i}=i+z'\right]\left(P_o^{r,\omega}\left(X_{\tau_n}=n+z\right)+c(r)^{n-(i+z')}\right)\label{yigitozludon}\\
&\leq E_o^{r,\omega}\left[\tau_i - \tau_{i-1}\right]\left(P_o^{r,\omega}\left(X_{\tau_n}=n+z\right)+c(r)^{n-(i+(B-1))}\right)\label{wissensie}
\end{align} where (\ref{yigitozludon}) follows from Lemma \ref{muhacir}. Recall Lemma \ref{mavisap} and see that
\begin{align*}
E_o^{r,\omega}\left[\left.\tau_n\,\right| X_{\tau_n}=n+z\right]&=\sum_{i=1}^nE_o^{r,\omega}\left[\left.\tau_i - \tau_{i-1}\,\right| X_{\tau_n}=n+z\right]\\
&\leq\sum_{i=1}^nE_o^{r,\omega}\left[\tau_i - \tau_{i-1}\right]\left(\frac{P_o^{r,\omega}\left(X_{\tau_n}=n+z\right)+c(r)^{n-(i+(B-1))}}{P_o^{r,\omega}\left(X_{\tau_n}=n+z\right)}\right)\\
&=E_o^{r,\omega}\left[\tau_n\right] + \frac{1}{P_o^{r,\omega}\left(X_{\tau_n}=n+z\right)}\sum_{i=1}^nc(r)^{n-(i+(B-1))}E_o^{r,\omega}\left[\tau_i - \tau_{i-1}\right]\\
&\leq E_o^{r,\omega}\left[\tau_n\right] + \frac{(\delta\mathrm{e}^r)^{-2B}\mathrm{e}^{-2r}}{G(T_n\omega)}\sum_{j=1}^n c(r)^{j-B}E_o^{r,T_{n-j}\omega}\left[\tau_1\right]\\
&\leq E_o^{r,\omega}\left[\tau_n\right] + \frac{(\delta\mathrm{e}^r)^{-2B}\mathrm{e}^{-2r}H_1(r)}{G(T_n\omega)}\sum_{j=1}^n c(r)^{j-B}\\
&\leq E_o^{r,\omega}\left[\tau_n\right] + \frac{(\delta\mathrm{e}^r)^{-2B}\mathrm{e}^{-2r}H_1(r)c(r)^{1-B}}{(1-c(r))G(T_n\omega)}\\
&= E_o^{r,\omega}\left[\tau_n\right] + \frac{W_1(r)}{G(T_n\omega)}.
\end{align*} This bound, in combination with  (\ref{turgayevlendi}), implies that 
$$\lambda_n'(r,\omega)-E_o^{r,\omega}\left[\tau_n\right]\leq\sup_{0\leq z<B}E_o^{r,\omega}\left[\left.\tau_n\,\right| X_{\tau_n}=n+z\right]-E_o^{r,\omega}\left[\tau_n\right]\leq\frac{W_1(r)}{G(T_n\omega)}.$$
The proof of the other direction is similar.
\end{proof}

\begin{lemma}
For every $r<r_c$, $n\geq1$ and $\mathbb{P}$-a.e.\ $\omega$,
\begin{equation}\label{cirkinmarina}
\lambda_n''(r,\omega)\leq\left(\frac{W_1(r)}{G(T_n\omega)}\right)^2 + n\left(\frac{W_2(r)+2H_1(r)W_1(r)}{G(T_n\omega)}\right)
\end{equation} is satisfied with $$W_2(r):=(\delta\mathrm{e}^r)^{-2B}\mathrm{e}^{-2r}\left(H_2(r)+\frac{6\left(H_1(r)\right)^2c(r)^{-2(B-1)}}{1-c(r)}\right),$$ $H_1(r)$ and $H_2(r)$ as in (\ref{boundoyleolmazboyleolur}), and $W_1(r)$ as in Lemma \ref{tornacihuso}.
\end{lemma}

\begin{proof}
For every $r<r_c$, $n\geq1$ and $\mathbb{P}$-a.e.\ $\omega$,
\begin{align}
\lambda_n''(r,\omega)&=\frac{E_o^\omega\left[\tau_n^2\mathrm{e}^{r\tau_n},\tau_n<\infty\right]}{E_o^\omega\left[\mathrm{e}^{r\tau_n},\tau_n<\infty\right]}-\left(\lambda_n'(r,\omega)\right)^2=\frac{E_o^{r,\omega}\left[\tau_n^2u_r\left(\omega,X_{\tau_n}\right)^{-1}\right]}{E_o^{r,\omega}\left[u_r\left(\omega,X_{\tau_n}\right)^{-1}\right]}-\left(\lambda_n'(r,\omega)\right)^2\nonumber\\
&=\frac{\sum_{z=0}^{B-1}E_o^{r,\omega}\left[\tau_n^2, X_{\tau_n}=n+z\right]u_r\left(\omega,n+z\right)^{-1}}{\sum_{z=0}^{B-1}P_o^{r,\omega}\left(X_{\tau_n}=n+z\right)u_r\left(\omega,n+z\right)^{-1}}-\left(\lambda_n'(r,\omega)\right)^2\nonumber\\
&\leq\sup_{0\leq z<B}E_o^{r,\omega}\left[\left.\tau_n^2\,\right| X_{\tau_n}=n+z\right] - \left(\lambda_n'(r,\omega)\right)^2.\label{lafalasol}
\end{align}
If $1\leq i<j\leq n$, then
\begin{align}
&E_o^{r,\omega}\left[(\tau_i-\tau_{i-1})(\tau_j-\tau_{j-1})\right]\nonumber\\
&\quad= E_o^{r,\omega}\left[(\tau_i-\tau_{i-1})(\tau_j-\tau_{j-1}), X_{\tau_{j-1}}=j-1\right]\nonumber\\
&\quad=E_o^{r,\omega}\left[\tau_i-\tau_{i-1}, X_{\tau_{j-1}}=j-1\right]E_{j-1}^{r,\omega}\left[\tau_j\right]\nonumber\\
&\quad\leq E_o^{r,\omega}\left[\tau_i-\tau_{i-1}\right]\left(P_o^{r,\omega}(X_{\tau_{j-1}}=j-1) + c(r)^{(j-1)-(i+(B-1))}\right)E_{j-1}^{r,\omega}\left[\tau_j\right]\label{wurden}\\
&\quad\leq E_o^{r,\omega}\left[\tau_i-\tau_{i-1}\right]E_o^{r,\omega}\left[\tau_j-\tau_{j-1}\right] +\left(H_1(r)\right)^2c(r)^{(j-1)-(i+(B-1))}\nonumber
\end{align} where (\ref{wurden}) follows from (\ref{wissensie}).

\noindent If $0\leq z<B$, then
\begin{align*}
&E_o^{r,\omega}\left[(\tau_i-\tau_{i-1})(\tau_j-\tau_{j-1}), X_{\tau_n}=n+z\right]\\
&\quad\leq E_o^{r,\omega}\left[(\tau_i-\tau_{i-1})(\tau_j-\tau_{j-1})\right]\left(P_o^{r,\omega}\left(X_{\tau_n}=n+z\right)+c(r)^{n-(j+(B-1))}\right)\\
&\quad\leq \left(E_o^{r,\omega}\left[\tau_i-\tau_{i-1}\right]E_o^{r,\omega}\left[\tau_j-\tau_{j-1}\right] +\left(H_1(r)\right)^2c(r)^{(j-1)-(i+(B-1))}\right)P_o^{r,\omega}\left(X_{\tau_n}=n+z\right)\\
&\quad\quad+\left(H_1(r)\right)^2\left(1+c(r)^{(j-1)-(i+(B-1))}\right)c(r)^{n-(j+(B-1))}.
\end{align*}
Therefore,
\begin{align*}
&E_o^{r,\omega}\left[\left.\tau_n^2\,\right| X_{\tau_n}=n+z\right]\\
&\quad=\sum_{i=1}^nE_o^{r,\omega}\left[\left.\left(\tau_i-\tau_{i-1}\right)^2\,\right| X_{\tau_n}=n+z\right]+2\sum_{j=1}^n\sum_{i=1}^{j-1}E_o^{r,\omega}\left[\left.\left(\tau_i-\tau_{i-1}\right)\left(\tau_j-\tau_{j-1}\right)\,\right| X_{\tau_n}=n+z\right]\\
&\quad\leq\frac{nH_2(r)}{P_o^{r,\omega}\left(X_{\tau_n}=n+z\right)}+2\sum_{j=1}^n\sum_{i=1}^{j-1}E_o^{r,\omega}\left[\tau_i-\tau_{i-1}\right]E_o^{r,\omega}\left[\tau_j-\tau_{j-1}\right]\\
&\quad\quad+2\left(H_1(r)\right)^2\sum_{j=1}^n\sum_{i=1}^{j-1}\left[c(r)^{(j-1)-(i+(B-1))}+\frac{c(r)^{n-(j+(B-1))}}{P_o^{r,\omega}\left(X_{\tau_n}=n+z\right)}\left(1+c(r)^{(j-1)-(i+(B-1))}\right)\right]\\
&\quad\leq\frac{nH_2(r)}{P_o^{r,\omega}\left(X_{\tau_n}=n+z\right)}+\left(E_o^{r,\omega}\left[\tau_n\right]\right)^2\\
&\quad\quad+\frac{2\left(H_1(r)\right)^2}{P_o^{r,\omega}\left(X_{\tau_n}=n+z\right)}\left[\frac{n\cdot c(r)^{-(B-1)}}{1-c(r)}+\frac{n\cdot c(r)^{-(B-1)}}{1-c(r)}+\frac{n\cdot c(r)^{-2(B-1)}}{1-c(r)}\right]\\
&\quad=\left(E_o^{r,\omega}\left[\tau_n\right]\right)^2+\frac{n}{P_o^{r,\omega}\left(X_{\tau_n}=n+z\right)}\left(H_2(r)+\frac{6\left(H_1(r)\right)^2c(r)^{-2(B-1)}}{1-c(r)}\right)\\
&\quad\leq\left(E_o^{r,\omega}\left[\tau_n\right]\right)^2+\frac{n\cdot(\delta\mathrm{e}^r)^{-2B}\mathrm{e}^{-2r}}{G(T_n\omega)}\left(H_2(r)+\frac{6\left(H_1(r)\right)^2c(r)^{-2(B-1)}}{1-c(r)}\right)\\
&\quad=\left(E_o^{r,\omega}\left[\tau_n\right]\right)^2+n\left(\frac{W_2(r)}{G(T_n\omega)}\right).
\end{align*}
Recall the bounds in (\ref{cirkinsonja}) and (\ref{lafalasol}), and conclude that
\begin{align*}
\lambda_n''(r,\omega)&\leq\sup_{0\leq z<B}E_o^{r,\omega}\left[\left.\tau_n^2\,\right| X_{\tau_n}=n+z\right] - \left(\lambda_n'(r,\omega)\right)^2\\
&=\sup_{0\leq z<B}E_o^{r,\omega}\left[\left.\tau_n^2\,\right| X_{\tau_n}=n+z\right] - \left(E_o^{r,\omega}\left[\tau_n\right]\right)^2 + \left(\left(E_o^{r,\omega}\left[\tau_n\right]\right)^2 - \left(\lambda_n'(r,\omega)\right)^2\right)\\
&\leq n\left(\frac{W_2(r)}{G(T_n\omega)}\right)+ \left[E_o^{r,\omega}\left[\tau_n\right] + \lambda_n'(r,\omega)\right]\left[E_o^{r,\omega}\left[\tau_n\right] - \lambda_n'(r,\omega)\right]\\
&\leq n\left(\frac{W_2(r)}{G(T_n\omega)}\right)+ \left(2nH_1(r)+\frac{W_1(r)}{G(T_n\omega)}\right)\frac{W_1(r)}{G(T_n\omega)}\\
&= \left(\frac{W_1(r)}{G(T_n\omega)}\right)^2 + n\left(\frac{W_2(r)+2H_1(r)W_1(r)}{G(T_n\omega)}\right).\qedhere
\end{align*}
\end{proof}

\begin{lemma}\label{sakinnazik}
$r\mapsto\lambda(r)$ is differentiable on $(-\infty,r_c)$ with $$\lambda'(r)=g(r)=\mathbb{E}\left[\lim_{x\to-\infty}P_x^{r,\omega}\left(X_{\tau_o}=0\right)E_o^{r,\omega}\left[\tau_1\right]\right].$$
\end{lemma}

\begin{proof}
Assumption (A1) implies that $\mathbb{P}\left(\omega: G(\omega)>0\right)=1$. Therefore, $\mathbb{P}\left(\omega: G(\omega)\geq\epsilon\right)\geq\frac{1}{2}$ for some $\epsilon>0$. For $\mathbb{P}$-a.e.\ $\omega$, there exists a sequence $(n_k)_{k\geq1} = (n_k(\omega))_{k\geq1}$ of integers such that $G(T_{n_k}\omega)\geq\epsilon$. (This follows from the ergodic theorem.) For every $r<r_c$, $k\geq1$ and $\mathbb{P}$-a.e.\ $\omega$,
$$\left|\lambda_{n_k}'(r,\omega)-E_o^{r,\omega}\left[\tau_{n_k}\right]\right|\leq\frac{W_1(r)}{G(T_{n_k}\omega)}\leq\epsilon^{-1}W_1(r)$$ by (\ref{cirkinsonja}). Thus,
$$\lim_{k\to\infty}\frac{1}{n_k}\lambda_{n_k}'(r,\omega)=\lim_{k\to\infty}\frac{1}{n_k}E_o^{r,\omega}\left[\tau_{n_k}\right]=g(r)=\mathbb{E}\left[\lim_{x\to-\infty}P_x^{r,\omega}\left(X_{\tau_o}=0\right)E_o^{r,\omega}\left[\tau_1\right]\right]$$ where the last two equalities follow from Lemma \ref{bugeceler}.

Given any $r<r_c$, pick $r_1, r_2\in\mathbb{R}$ such that $r_1<r<r_2<r_c$. For $\mathbb{P}$-a.e.\ $\omega$, note that
$$\sup_{r_1\leq s\leq r_2 \atop k\geq1}\frac{1}{n_k}\lambda_{n_k}'(s,\omega)\leq\sup_{r_1\leq s\leq r_2 \atop k\geq1}\left(H_1(s) + \frac{1}{n_k}\epsilon^{-1}W_1(s)\right)<\infty.$$ Therefore, the bounded convergence theorem implies that $$\lambda(r)-\lambda(r_1)=\lim_{k\to\infty}\frac{1}{n_k}\left(\lambda_{n_k}(r,\omega)-\lambda_{n_k}(r_1,\omega)\right)=\lim_{k\to\infty}\int_{r_1}^r\frac{1}{n_k}\lambda_{n_k}'(s,\omega)\mathrm{d}s=\int_{r_1}^rg(s)\mathrm{d}s.$$ It is easy to see that $g(\cdot)$ is Lipschitz continuous at $r$ since $$\sup_{r_1\leq s\leq r_2 \atop k\geq1}\frac{1}{n_k}\lambda_{n_k}''(s,\omega)<\infty$$ by (\ref{cirkinmarina}). The desired result follows from the fundamental theorem of calculus. 
\end{proof}




\subsection{Verification of the Ansatz}\label{fluxdensity}

\begin{lemma}\label{cakabeycan}
For every $r<r_c$, $$P_o^r\left(\lim_{n\to\infty}\frac{\tau_n}{n}=\lambda'(r)\right)=1.$$
\end{lemma}

\begin{proof}
For every $r<r_c$, $s<r_c-r$ and $\mathbb{P}$-a.e.\ $\omega$,
\begin{align*}
\lim_{n\to\infty}\frac{1}{n}\log E_o^{r,\omega}\left[\mathrm{e}^{s\tau_n}\right]&=\lim_{n\to\infty}\frac{1}{n}\log E_o^\omega\left[\mathrm{e}^{(r+s)\tau_n}u_r\left(\omega,X_{\tau_n}\right),\tau_n<\infty\right]\\
&=\lim_{n\to\infty}\frac{1}{n}\log E_o^\omega\left[\mathrm{e}^{(r+s)\tau_n},\tau_n<\infty\right] + \lim_{n\to\infty}\frac{1}{n}\log u_r\left(\omega,n\right)\\
&=\lambda(r+s)-\lambda(r)
\end{align*} where the last equality follows from Lemma \ref{ongbak}.

For every $\epsilon>0$, a standard application of Chebyshev's inequality shows that
\begin{align*}
\limsup_{n\to\infty}\frac{1}{n}\log P_o^{r,\omega}\left(\frac{\tau_n}{n}-\lambda'(r)>\epsilon\right)&\leq\limsup_{n\to\infty}\frac{1}{n}\log E_o^{r,\omega}\left[\mathrm{e}^{s\tau_n}\right] - s\left(\lambda'(r)+\epsilon\right)\\
&=\lambda(r+s)-\lambda(r)-s\left(\lambda'(r)+\epsilon\right)<0
\end{align*} when $s>0$ is small enough. Similarly, $$\limsup_{n\to\infty}\frac{1}{n}\log P_o^{r,\omega}\left(\frac{\tau_n}{n}-\lambda'(r)<-\epsilon\right)<0\quad\mbox{and}\quad\limsup_{n\to\infty}\frac{1}{n}\log P_o^{r,\omega}\left(\left|\frac{\tau_n}{n}-\lambda'(r)\right|>\epsilon\right)<0.$$ Since $\epsilon>0$ is arbitrary, the Borel--Cantelli lemma implies the desired result.
\end{proof}

\begin{lemma}
$F_r:\Omega\times\mathcal{R}\to\mathbb{R}$, defined by $F_r(\omega,z):=\log u_r(\omega,z) + z\lambda(r)$ for each $z\in\mathcal{R}$, is in class $\mathcal{K}$.
\end{lemma}

\begin{proof}
For each $z\in\mathcal{R}$ and $\mathbb{P}$-a.e.\ $\omega$, $$\left|F_r(\omega,z)\right|\leq\left|\log u_r(\omega,z)\right| + \left|z\lambda(r)\right|\leq B\left(-\log\left(\delta\mathrm{e}^r\right)+\left|\lambda(r)\right|\right)<\infty.$$ Therefore, $F_r$ satisfies the moment condition of Definition \ref{K}. The closed loop condition follows immediately from (\ref{cangorecek}). Finally, if $1\leq z\leq B$, then
$$\mathbb{E}\left[\log u_r(\omega,z)\right]=\mathbb{E}\left[\log\left(\prod_{i=0}^{z-1}u_r(T_i\omega,1)\right)\right]= \sum_{i=0}^{z-1}\mathbb{E}\left[\log u_r(T_i\omega,1)\right]=-z\lambda(r).$$
(The case $-B\leq z\leq -1$ is similar.) Hence, $F_r$ satisfies the mean zero condition as well.
\end{proof}

It follows easily from Lemma \ref{cakabeycan} that the LLN for the mean velocity of the particle holds with limiting velocity $(\lambda'(r))^{-1}$.
If there exists a $\phi_r\in L^1(\mathbb{P})$ such that $\phi_r\,\mathrm{d}\mathbb{P}$ is a $\hat{\pi}_r$-invariant probability measure, then $\mu_\xi\in M_1(\Omega\times\mathcal{R})$ with 
\begin{equation}\label{selimgelirmi}
\mathrm{d}\mu_\xi(\omega,z):=\hat{\pi}_r(\omega,z)\phi_r(\omega)\mathrm{d}\mathbb{P}(\omega)=\pi(0,z)\mathrm{e}^{-z\lambda(r)+F_r(\omega,z)+r}\phi_r(\omega)\mathrm{d}\mathbb{P}(\omega)
\end{equation} fits the Ansatz given in Lemma \ref{lagrange} for $\xi=(\lambda'(r))^{-1}$. The existence of such a $\phi_r$ is a corollary of the following general result which completes our construction.

\begin{theorem}\label{density}
Suppose $d=1$. If an environment kernel $\hat{\pi}:\Omega\times\mathcal{R}\to\mathbb{R}^+$ satisfies $\mathbb{P}\left(\omega: \hat{\pi}(\omega,\pm1)\geq\epsilon\right)=1$ for some $\epsilon>0$, and if $E_o^{\hat{\pi}}[\tau_1]<\infty$, then the following hold:
\begin{itemize}
\item[(a)] $\phi(\omega):=\lim_{x\rightarrow-\infty} E_x^{\hat{\pi},\omega}\left[\sum_{k=0}^\infty\one_{X_k=0}\right]\geq\epsilon^B$ exists for $\mathbb{P}$-a.e.\ $\omega$.
\item[(b)] $\phi\in L^1(\mathbb{P})$.
\item[(c)] $\mathbb{Q}\in M_1(\Omega)$, defined by $\mathrm{d}\mathbb{Q}(\omega):=\left({1}/{\left\|\phi\right\|_{L^1(\mathbb{P})}}\right)\phi(\omega)\mathrm{d}\mathbb{P}(\omega)$, is $\hat{\pi}$-invariant.
\end{itemize}
\end{theorem}
\begin{remark}
Br\'emont \cite{Bremont} also shows the existence of a $\hat{\pi}$-invariant probability measure $\mathbb{Q}\ll\mathbb{P}$ in the case of ballistic random walk with bounded jumps on $\mathbb{Z}$. However, his argument is not elementary, assumes a stronger ellipticity condition, and does not provide a formula for the density. Rassoul-Agha \cite{Firas} takes an approach similar to ours, but resorts to Ces\`aro means and weak limits instead of showing the almost sure convergence in part (a) of Theorem \ref{density}, and assumes that the so-called Kalikow condition holds. For the related model of ``random walk on a strip", Roitershtein \cite{roiter} shows the existence of the ergodic invariant measure. It is easy to see that the natural analog of our formula works in that setting.
\end{remark}

\begin{proof}[Proof of Theorem \ref{density}]

Consider the hitting time $V_o:=\inf\{k\geq0:\,X_k=0\}$. For every $x\in\mathbb{Z}$ and $\omega\in\Omega$, let $\psi(\omega,x):=P_x^{\hat{\pi},\omega}(V_o<\infty)$. If $x\neq0$, then \[\psi(\omega,x)=\sum_{z\in\mathcal{R}}\hat{\pi}(T_x\omega,z)\psi(\omega,x+z).\] The function $\phi(\omega,x):=E_x^{\hat{\pi},\omega}\left[\sum_{k=0}^\infty\one_{X_k=0}\right]$ clearly satisfies $\phi(\omega,x)=\phi(\omega,0)\psi(\omega,x)$. Hence, \[\phi(\omega):=\lim_{x\rightarrow-\infty}\phi(\omega,x)=\phi(\omega,0)\lim_{x\rightarrow-\infty}\psi(\omega,x)\] exists for $\mathbb{P}$-a.e.\ $\omega$ by Lemma \ref{muhacir}. The ellipticity condition implies that $\mathbb{P}\left(\omega: \phi(\omega)\geq\epsilon^B\right)=1$. This proves part (a) of the theorem.

Let us now show that $\phi\in L^1(\mathbb{P})$. For every $n\geq1$ and $\mathbb{P}$-a.e.\ $\omega$,
\begin{align}
\sum_{i=0}^{n-1}\phi(T_i\omega)&=\sum_{i=0}^{n-1}\lim_{x\rightarrow-\infty}E_x^{\hat{\pi},T_i\omega}\left[\sum_{k=0}^\infty\one_{X_k=0}\right]=\sum_{i=0}^{n-1}\lim_{x\rightarrow-\infty}E_x^{\hat{\pi},\omega}\left[\sum_{k=0}^\infty\one_{X_k=i}\right]\nonumber\\
&=\lim_{x\rightarrow-\infty}E_x^{\hat{\pi},\omega}\left[\#\{k\geq0:\,0\leq X_k\leq n-1\}\right]\nonumber\\
&\leq\lim_{x\rightarrow-\infty}E_x^{\hat{\pi},\omega}\left[\tau_n-\tau_o\right]+\lim_{x\rightarrow-\infty}E_x^{\hat{\pi},\omega}\left[\#\{k\geq \tau_n:\,X_k\leq n-1\}\right]\nonumber\\
&=\lim_{x\rightarrow-\infty}E_x^{\hat{\pi},\omega}\left[\tau_n-\tau_o\right]+\lim_{x\rightarrow-\infty}E_x^{\hat{\pi},T_n\omega}\left[\#\{k\geq \tau_o:\,X_k\leq -1\}\right]\nonumber\\
&\leq\lim_{x\rightarrow-\infty}E_x^{\hat{\pi},\omega}\left[\tau_n-\tau_o\right]+\sup_{0\leq z<B}E_z^{\hat{\pi},T_n\omega}\left[\#\{k\geq0:\,X_k\leq -1\}\right].\label{asil}
\end{align} Here, $\#$ denotes the number of elements of a set. If $0\leq z<B$, then
\begin{align*}
&E_z^{\hat{\pi},\omega}\left[\#\{k\geq0:\,X_k\leq-1\}\right]=E_z^{\hat{\pi},\omega}\left[\#\{k\geq0:\,X_k\leq-1\},\bar{\tau}_{-1}<\infty\right]\\
&\quad\quad=\sum_{z'=-B}^{-1}P_z^{\hat{\pi},\omega}\left(\bar{\tau}_{-1}<\infty,X_{\bar{\tau}_{-1}}=z'\right)E_{z'}^{\hat{\pi},\omega}\left[\#\{k\geq0:\,X_k\leq-1\}\right]\\
&\quad\quad=\sum_{z'=-B}^{-1}P_z^{\hat{\pi},\omega}\left(\bar{\tau}_{-1}<\infty,X_{\bar{\tau}_{-1}}=z'\right)\left(E_{z'}^{\hat{\pi},\omega}\left[\tau_o\right]+E_{z'}^{\hat{\pi},\omega}\left[\#\{k\geq\tau_o:\,X_k\leq-1\}\right]\right)\\
&\quad\quad\leq\sum_{z'=-B}^{-1}P_z^{\hat{\pi},\omega}\left(\bar{\tau}_{-1}<\infty,X_{\bar{\tau}_{-1}}=z'\right)\left(E_{z'}^{\hat{\pi},\omega}\left[\tau_o\right]+\sup_{0\leq z^{''}<B}E_{z^{''}}^{\hat{\pi},\omega}\left[\#\{k\geq0:\,X_k\leq-1\}\right]\right)\\
&\quad\quad\leq P_z^{\hat{\pi},\omega}\left(\bar{\tau}_{-1}<\infty\right)\left(\sup_{-B\leq z'\leq-1}E_{z'}^{\hat{\pi},\omega}\left[\tau_o\right]+\sup_{0\leq z^{''}<B}E_{z^{''}}^{\hat{\pi},\omega}\left[\#\{k\geq0:\,X_k\leq-1\}\right]\right).
\end{align*}
Therefore, $$\sup_{0\leq z<B}E_z^{\hat{\pi},\omega}\left[\#\{k\geq0:\,X_k\leq -1\}\right]\leq\left(\frac{\sup_{0\leq z< B}P_z^{\hat{\pi},\omega}(\bar{\tau}_{-1}<\infty)}{1-\sup_{0\leq z< B}P_z^{\hat{\pi},\omega}(\bar{\tau}_{-1}<\infty)}\right)\sup_{-B\leq z'\leq-1}E_{z'}^{\hat{\pi},\omega}\left[\tau_o\right]=:D(\omega).$$
Since $\mathbb{P}\left(\omega: D(\omega)<\infty\right)=1$, there exists a $C<\infty$ such that $\mathbb{P}\left(\omega: D(\omega)\leq C\right)\geq\frac{1}{2}$. For $\mathbb{P}$-a.e.\ $\omega$, there exists a sequence $(n_j)_{j\geq1}=(n_j(\omega))_{j\geq1}$ of integers such that $D(T_{n_j}\omega)\leq C$. (This follows from the ergodic theorem.) By (\ref{asil}) and the ergodic theorem,

\begin{align*}
\left\|\phi\right\|_{L^1(\mathbb{P})}&=\lim_{j\rightarrow\infty}\frac{1}{n_j}\sum_{i=0}^{n_j-1}\phi(T_i\omega)\leq\lim_{j\rightarrow\infty}\frac{1}{n_j}\lim_{x\rightarrow-\infty}E_x^{\hat{\pi},\omega}\left[\tau_{n_j}-\tau_o\right]\\
&=\lim_{j\rightarrow\infty}\frac{1}{n_j}\lim_{x\rightarrow-\infty}\sum_{i=0}^{n_j-1}E_x^{\hat{\pi},\omega}\left[\tau_{i+1}-\tau_i\right]\leq\lim_{j\rightarrow\infty}\frac{1}{n_j}\sum_{i=0}^{n_j-1}E_o^{\hat{\pi},T_i\omega}\left[\tau_1\right]=E_o^{\hat{\pi}}\left[\tau_1\right]<\infty.
\end{align*} This proves part (b) of the theorem.

For every $x\neq0$ and $\mathbb{P}$-a.e.\ $\omega$, note that
\begin{align*}
&\sum_{z\in\mathcal{R}}E_{x+z}^{\hat{\pi},T_{-z}\omega}\left[\sum_{k=0}^\infty\one_{X_k=0}\right]\hat{\pi}(T_{-z}\omega,z)=\sum_{z\in\mathcal{R}}E_x^{\hat{\pi},\omega}\left[\sum_{k=0}^\infty\one_{X_k=-z}\right]\hat{\pi}(T_{-z}\omega,z)\\&=E_x^{\hat{\pi},\omega}\left[\sum_{k=0}^\infty\one_{X_{k+1}=0}\right]=E_x^{\hat{\pi},\omega}\left[\sum_{k=0}^\infty\one_{X_{k}=0}\right].
\end{align*}
Let $x\to-\infty$ and conclude that $$\sum_{z\in\mathcal{R}}\phi(T_{-z}\omega)\hat{\pi}(T_{-z}\omega,z)=\phi(\omega).$$ This proves part (c) of the theorem.
\end{proof}

\subsection{Explicit formula for the rate function}\label{formulabir}

\begin{proof}[Proof of Lemma \ref{lifeisrandom}]
For every $r<r_c$, $n\geq B+1$ and $\mathbb{P}$-a.e.\ $\omega$, $$\left(\delta\mathrm{e}^r\right)E_o^\omega\left[\mathrm{e}^{r\tau_1},\tau_1<\infty\right]\leq u_{r,n}(\omega,1)E_o^\omega\left[\mathrm{e}^{r\tau_1},\tau_1<\infty\right]\leq\left(\delta\mathrm{e}^r\right)^{-(B-1)}$$ where the first and the second inequalities follow from (\ref{isilam}) and (\ref{deyyus}), respectively. Thus, $$\mathbb{P}\left(\omega: E_o^\omega\left[\mathrm{e}^{r\tau_1},\tau_1<\infty\right]\leq\left(\delta\mathrm{e}^r\right)^{-B}\right)=1$$ for $r<r_c$, and also for $r=r_c$ by the monotone convergence theorem. Lemma \ref{ongbak} and (\ref{cangorecek}) are clearly valid for $r\leq r_c$, and $$\lambda(r):=\lim_{n\to\infty}\frac{1}{n}\log E_o^\omega\left[\mathrm{e}^{r\tau_n},\tau_n<\infty\right]=-\mathbb{E}\left[\log u_r(\cdot,1)\right]\leq-\log\left(\delta\mathrm{e}^r\right)<\infty.$$

Suppose $r>r_c$. Then, $E_o^\omega\left[\mathrm{e}^{r\tau_1},\tau_1<\infty\right]=\infty$ for $\mathbb{P}$-a.e.\ $\omega$. For every $n\geq B$,
\begin{align*}
E_o^\omega\left[\mathrm{e}^{r\tau_n},\tau_n<\infty\right]&=\sum_{z=1}^BE_o^\omega\left[\mathrm{e}^{r\tau_1},\tau_1<\infty,X_{\tau_1}=z\right]E_z^\omega\left[\mathrm{e}^{r\tau_n},\tau_n<\infty\right]\\
&\geq\sum_{z=1}^BE_o^\omega\left[\mathrm{e}^{r\tau_1},\tau_1<\infty,X_{\tau_1}=z\right]\left(\delta\mathrm{e}^r\right)^{n-z}\\
&\geq E_o^\omega\left[\mathrm{e}^{r\tau_1},\tau_1<\infty\right]\left(\delta\mathrm{e}^r\right)^{n-1}=\infty.
\end{align*}
Therefore, $\lambda(r):=\lim_{n\to\infty}\frac{1}{n}\log E_o^\omega\left[\mathrm{e}^{r\tau_n},\tau_n<\infty\right]=\infty$. This proves that $r\mapsto\lambda(r)$ is (i) deterministic, and (ii) finite precisely on $(-\infty,r_c]$. Note that $0\leq r_c\leq-\log\delta<\infty$.

The function $r\mapsto\lambda(r)$ is differentiable on $(-\infty,r_c)$ by Lemma \ref{sakinnazik}. Suppose there exist $r_1<r_c$ and $r_2<r_c$ such that $\lambda'(r_1)=\lambda'(r_2)$. Then, for $r=r_1$, the measure $\mu_\xi$ (defined in (\ref{selimgelirmi})) fits the Ansatz given in Lemma \ref{lagrange} for $\xi=(\lambda'(r_1))^{-1}$. The same is true for $r=r_2$. However, such a $\mu_\xi$ is unique by Lemma \ref{lagrange}. Therefore, $\mathbb{P}\left(\omega:u_{r_1}(\omega,1)=u_{r_2}(\omega,1)\right)=1$, $\lambda(r_1)=\lambda(r_2)$ and $r_1=r_2$. This proves that $r\mapsto\lambda(r)$ is strictly convex on $(-\infty,r_c)$.

For any $r<r_c$, Lemma \ref{cakabeycan} says that $P_o^r\left(\lim_{n\to\infty}\frac{\tau_n}{n}=\lambda'(r)\right)=1$. The function $r\mapsto\lambda'(r)$ is strictly increasing and the jumps of the walk under $P_o^r$ are bounded by $B$. Therefore, $\xi_c^{-1}=\lambda'(r_c-)>\lambda'(r)\geq B^{-1}$.
We have proved half of Lemma \ref{lifeisrandom}, namely the statements involving $r\mapsto\lambda(r)$. As usual, we leave the proof of the other half to the reader.

What remains to be shown is that the same $r_c$ works for $\lambda(\cdot)$ and $\bar{\lambda}(\cdot)$. This is proved in Appendix C.
\end{proof}

\begin{proof}[Proof of Theorem \ref{explicitformulah}]
For every $r<r_c$,
\begin{align}
\lambda(r)&=\lim_{n\to\infty}\frac{1}{nB}\log E_o^\omega\left[\mathrm{e}^{r\tau_{nB}},\tau_{nB}<\infty\right]\geq\lim_{n\to\infty}\frac{1}{nB}\log E_o^\omega\left[\mathrm{e}^{r\tau_{nB}},X_{n+1}=nB\right]\label{kafatek}\\
&\geq\lim_{n\to\infty}\frac{1}{nB}\log\left(\mathrm{e}^{rn-|r|}P_o^\omega\left(X_{n+1}=nB\right)\right)=B^{-1}\left(r-I(B)\right).\label{kafacift}
\end{align}
In (\ref{kafatek}), $X_{n+1}$ is used instead of $X_n$ in order to avoid problems when $\mathbb{P}$ is not ergodic under $T_B$ (e.g.\ when the environment is $B$-periodic.)
The function $r\mapsto\lambda(r)$ is strictly convex and differentiable on $(-\infty,r_c)$ by Lemma \ref{lifeisrandom}.
Since $\lambda'(r)\geq B^{-1}$, (\ref{kafacift}) implies that $\lim_{r\to-\infty}\lambda'(r)=B^{-1}$.

For every $\xi\in(\xi_c,B)$, there exists a unique $r=r(\xi)\in(-\infty,r_c)$ such that $\xi^{-1}=\lambda'(r)$. Lemma \ref{lagrange} implies that the measure $\mu_\xi$ (given in (\ref{selimgelirmi})) is the unique minimizer of (\ref{level1ratetilde}). Therefore, $$I(\xi) = \mathfrak{I}(\mu_\xi)=r(\xi)-\xi\lambda(r(\xi))$$ by (\ref{sifirladik}). Since $\lambda'(r(\xi))=\xi^{-1}$, it is clear that $$I(\xi)=\sup_{r\in\mathbb{R}}\left\{r-\xi\lambda(r)\right\}=\xi\sup_{r\in\mathbb{R}}\left\{r\xi^{-1}-\lambda(r)\right\}=\xi\lambda^{*}(\xi^{-1}).$$
By convex duality, $\xi\mapsto I(\xi)$ is strictly convex and differentiable on $(\xi_c,B)$.

If $\xi_c=0$, then we have identified $I(\cdot)$ on $(0,B)$. Let us now suppose $\xi_c>0$. Note that $$I'(\xi)=\frac{\mathrm{d}}{\mathrm{d}\xi}[r(\xi)-\xi\lambda(r(\xi))]=r'(\xi)-\lambda(r(\xi))-\xi\lambda'(r(\xi))r'(\xi)=-\lambda(r(\xi)).$$ Therefore, $I(\xi_c)-\xi_cI'(\xi_c+)=r_c$. This implies by convexity that $I(0)\geq r_c$. On the other hand, $$E_o^\omega[\mathrm{e}^{r\tau_1},\tau_1<\infty]=\sum_{k=1}^{\infty}\mathrm{e}^{rk}P_o^\omega(\tau_1=k)\leq\sum_{k=1}^{\infty}\mathrm{e}^{rk}P_o^\omega(1\leq X_k\leq B)\leq\sum_{k=1}^{\infty}\mathrm{e}^{(r-I(0))k + o(k)}<\infty$$ for any $r<I(0)$. Hence, $r_c = I(0)$. The equality $I(\xi_c)-\xi_cI'(\xi_c+)=I(0)$ forces $I(\cdot)$ to be affine linear on $[0,\xi_c]$ with a slope of $I'(\xi_c+)$. In particular, $\xi\mapsto I(\xi)$ is differentiable on $(0,B)$.

Still supposing $\xi_c>0$, fix $\xi\in(0,\xi_c]$. Then, $\frac{\mathrm{d}}{\mathrm{d}r}\left(r-\xi\lambda(r)\right)>0$ for every $r<r_c$. Therefore, $$\sup_{r\in\mathbb{R}}\left\{r-\xi\lambda(r)\right\}=r_c-\xi\lambda(r_c)=I(0)+\xi I'(\xi_c+)=I(\xi).$$ In short, $I(\xi)=\sup_{r\in\mathbb{R}}\left\{r-\xi\lambda(r)\right\}$ for every $\xi\in(0,B)$.

Let us no longer suppose $\xi_c>0$. At $\xi=B$,
\begin{align*}
I(B)&=\lim_{\xi\to B-}I(\xi)=\lim_{\xi\to B-}\left[r(\xi)-\xi\lambda(r(\xi))\right]=\lim_{r\to-\infty}\left[r-\frac{\lambda(r)}{\lambda'(r)}\right]\\
&\leq\lim_{r\to-\infty}\left[r-B\lambda(r)\right]\leq\sup_{r\in\mathbb{R}}\left\{r-B\lambda(r)\right\}\leq I(B).
\end{align*}
Here, the last inequality follows from (\ref{kafacift}). It is easy to check that $I(\xi)=\sup_{r\in\mathbb{R}}\left\{r-\xi\lambda(r)\right\}=\infty$ when $\xi>B$. This concludes the proof of Theorem \ref{explicitformulah} for $\xi\geq0$. The arguments regarding $\xi<0$ are similar.
\end{proof}

\begin{proof}[Proof of Proposition \ref{Wronskian}]
Suppose that the walk is nearest-neighbor. For every $\omega\in\Omega$, let $\rho(\omega):=\frac{\pi(0,-1)}{\pi(0,1)}$. For every $r<r_c$, recall that the function $u_r:\Omega\times\mathbb{Z}\to\mathbb{R}^+$ satisfies
\begin{equation}\label{saykool}
u_r(\omega,x)=\pi(x,x-1)\mathrm{e}^ru_r(\omega,x-1)+\pi(x,x+1)\mathrm{e}^ru_r(\omega,x+1)
\end{equation} and that $\lambda(r)=-\mathbb{E}\left[\log u_r(\cdot,1)\right]$. Replacing $(\tau_n)_{n\geq1}$ by $(\bar{\tau}_{-n})_{n\geq1}$ in the whole construction, one can similarly obtain a function $\bar{u}_r:\Omega\times\mathbb{Z}\to\mathbb{R}^+$ such that
\begin{equation}\label{saykoolmak}
\bar{u}_r(\omega,x)=\pi(x,x-1)\mathrm{e}^r\bar{u}_r(\omega,x-1)+\pi(x,x+1)\mathrm{e}^r\bar{u}_r(\omega,x+1)
\end{equation}
and $\bar{\lambda}(r)=\mathbb{E}\left[\log \bar{u}_r(\cdot,1)\right]$. Introduce
$$U_r(\omega,x):=\left(
\begin{array}{ll}
u_r(\omega,x+1)&\bar{u}_r(\omega,x+1)\\
u_r(\omega,x)&\bar{u}_r(\omega,x)
\end{array}
\right)$$ and the Wronskian $W_r(\omega,x) := \mathrm{det}\left(U_r(\omega,x)\right)$. 
By (\ref{saykool}) and (\ref{saykoolmak}),
$$U_r(\omega,x)=\left(
\begin{array}{ll}
\mathrm{e}^{-r}(1+\rho(T_x\omega))&-\rho(T_x\omega)\\
1&0
\end{array}
\right)U_r(\omega,x-1)$$ and $W_r(\omega,x)=\rho(T_x\omega)W_r(\omega,x-1)$. However, it follows from (\ref{cangorecek}) that
$$W_r(\omega,x)=u_r(\omega,x+1)\bar{u}_r(\omega,x)-u_r(\omega,x)\bar{u}_r(\omega,x+1)=u_r(\omega,x)\bar{u}_r(\omega,x)W_r(T_x\omega,0).$$
Therefore, at $x=1$,
\begin{align*}
\rho(T_1\omega)W_r(\omega,0)&=W_r(\omega,1)=u_r(\omega,1)\bar{u}_r(\omega,1)W_r(T_1\omega,0),\\
\log\rho(T_1\omega)+\log W_r(\omega,0)&=\log u_r(\omega,1)+\log\bar{u}_r(\omega,1)+\log W_r(T_1\omega,0).
\end{align*} Take $\mathbb{E}$-expectation to deduce that $\mathbb{E}\left[\log\rho(\cdot)\right]=\bar{\lambda}(r)-\lambda(r)$. In particular, $\bar{\lambda}(r)-\lambda(r)=\bar{\lambda}(0)-\lambda(0)$.

Hence, for every $\xi\in[-1,0)$, 
$$I(\xi)=\sup_{r\in\mathbb{R}}\left\{r+\xi\bar{\lambda}(r)\right\}=\sup_{r\in\mathbb{R}}\left\{r-(-\xi)\lambda(r)\right\}+\xi\cdot\mathbb{E}\left[\log\rho(\cdot)\right]=I(-\xi)+\xi\cdot\mathbb{E}\left[\log\rho(\cdot)\right].$$

In order to prove that such a symmetry is generally absent for walks with bounded jumps, let us provide a counterexample. Consider classical random walk on $\mathbb{Z}$. Let $p(z):=P_o(X_1=z)$ for every $z\in\mathbb{Z}$. Suppose $p(-2)=1/7$, $p(-1)=3/7$, $p(1)=1/7$ and $p(2)=2/7$. For $r<0$, it is easy to see that $\mathrm{e}^{-\lambda(r)}$ and $\mathrm{e}^{\bar{\lambda}(r)}$ are the two positive roots $x_r$ and $\bar{x}_r$ of the polynomial $2x^4 + x^3 -7\mathrm{e}^{-r}x^2 + 3x + 1$. By plugging in various values for $r$, one can check that $\bar{\lambda}(r)-\lambda(r)=\log(x_r\bar{x}_r)$ is not independent of $r$.
\end{proof}

\section*{Appendix A}

\begin{proposition}
For nearest-neighbor random walk on $\mathbb{Z}$ in a uniformly elliptic product environment, the function $\mathfrak{I}:M_1(\Omega\times\mathcal{R})\to\mathbb{R}^+$, given by (\ref{level2ratetilde}), is not lower semicontinuous. Hence, $\mathfrak{I}\neq\mathfrak{I}^{**}$.
\end{proposition}

\begin{proof}
Define $a_{\infty}:=\mathbb{E}[\rho]^{-1/2}$ where $\rho(\omega) := {\pi(0,-1)}/{\pi(0,1)}$ for every $\omega\in\Omega$. Given any sequence $(a_n)_{n\geq1}$ that is strictly increasing to $a_{\infty}$, introduce a sequence $(\hat{\pi}_n)_{n\geq1}$ of environment kernels by setting $$\hat{\pi}_n(\omega,-1) := \frac{a_n\pi(0,-1)}{a_n\pi(0,-1) + a_n^{-1}\pi(0,1)}\qquad\mbox{and}\qquad\hat{\pi}_n(\omega,1) := \frac{a_n^{-1}\pi(0,1)}{a_n\pi(0,-1) + a_n^{-1}\pi(0,1)}$$ for $1\leq n\leq\infty$.

When $1\leq n<\infty$, $\rho_n(\omega) := {\hat{\pi}_n(\omega,-1)}/{\hat{\pi}_n(\omega,1)}$ satisfies $\mathbb{E}[\rho_n]=a_n^2\mathbb{E}[\rho]={a_n^2}/{a_{\infty}^2}<1$. It follows from \cite{Solomon} that the LLN holds under the environment kernel $\hat{\pi}_n$, and the limiting velocity is positive. By \cite{alili}, there exists a $\hat{\pi}_n$-invariant probability measure $\mathbb{Q}_n\ll\mathbb{P}$. Let us define $\mu_n\in M_1(\Omega\times U)$ by $\mathrm{d}\mu_n(\omega,z):=\mathrm{d}\mathbb{Q}_n(\omega)\hat{\pi}_n(\omega,z)$. Then, $\mu_n\in M_{1,s}^{\ll}(\Omega\times U)$.

The case $n=\infty$ is different since $\rho_\infty(\omega) := {\hat{\pi}_\infty(\omega,-1)}/{\hat{\pi}_\infty(\omega,1)}$ satisfies $\mathbb{E}[\rho_\infty]=1$. By Jensen's inequality, $\mathbb{E}[\log\rho_\infty]<\log\mathbb{E}[\rho_\infty]=0$. Therefore, the walk under the environment kernel $\hat{\pi}_\infty$ is transient to the right, but the limiting velocity is zero. (See \cite{Solomon}.)

$M_1(\Omega)$ is weakly compact. There exists a subsequence $(\mathbb{Q}_{n_k})_{k\geq1}$ of $(\mathbb{Q}_n)_{n\geq1}$ that converges to some $\mathbb{Q}_\infty\in M_1(\Omega)$. Define $\mu_\infty\in M_1(\Omega\times U)$ by $\mathrm{d}\mu_\infty(\omega,z):=\mathrm{d}\mathbb{Q}_\infty(\omega)\hat{\pi}_\infty(\omega,z)$. Clearly, $\mu_{n_k}$ converges weakly to $\mu_\infty$. Also, $(\mu_\infty)^1=(\mu_\infty)^2=\mathbb{Q}_\infty$, i.e., $\mathbb{Q}_\infty$ is $\hat{\pi}_\infty$-invariant. However, since the walk under the environment kernel $\hat{\pi}_\infty$ is transient but not ballistic, $\mathbb{Q}_\infty$ is not absolutely continuous relative to $\mathbb{P}$. (See \cite{Bremont}.) Therefore, $\mu_\infty\not\in M_{1,s}^{\ll}(\Omega\times U)$. By (\ref{level2ratetilde}), $\mathfrak{I}(\mu_\infty)=\infty$. On the other hand, it is easy to see that $$\lim_{k\to\infty}\mathfrak{I}(\mu_{n_k})=\int\sum_{z\in U}\hat{\pi}_\infty(\omega,z)\log\frac{\hat{\pi}_\infty(\omega,z)}{\pi(0,z)}\mathrm{d}\mathbb{Q}_\infty(\omega)$$ which is finite by the uniform ellipticity assumption. This proves that $\mathfrak{I}$ is not lower semicontinuous.
\end{proof}

\begin{remark}
In the case of random walk on $\mathbb{Z}^d$ in an elliptic periodic environment, $\Omega$ has finitely many elements. Therefore, $M_1(\Omega\times\mathcal{R})$ is finite-dimensional. Ellipticity ensures that $\mathfrak{I}$ is finite on $M_1(\Omega\times\mathcal{R})$. Note that a convex function on a finite-dimensional space is continuous whenever it is finite. Hence, $\mathfrak{I}$ is continuous on $M_1(\Omega\times\mathcal{R})$.
\end{remark}

\section*{Appendix B}

\begin{proof} [Sketch of the proof of Lemma \ref{GRR}]
For every $F\in\mathcal{K}$, $y\in\mathbb{Z}^d$ and $\omega\in\Omega$, let $f(\omega,y):=\sum_{i=0}^{j-1}F(T_{y_i}\omega,y_{i+1}-y_i)$ where $(y_{i})_{i=0}^j$ is any sequence in $\mathbb{Z}^d$ with $y_o=0$, $y_j=y$ and $y_{i+1}-y_i\in\mathcal{R}$. The closed loop condition (given in Definition \ref{K}) ensures that $f:\Omega\times\mathbb{Z}^d\to\mathbb{R}$ is well defined. Extend $f$ to $\Omega\times\mathbb{R}^d$ via interpolation. For every $n\geq1$, define $g_n:\mathbb{R}^d\to\mathbb{R}$ by $g_n(t):=f(\omega,nt)/n$.

The crucial step is to show that $(g_n)_{n\geq1}$ is equicontinuous and hence compact. This is accomplished by estimating the modulus of continuity of $g_n$ from the moment condition in Definition \ref{K} via a theorem of Garsia, Rodemich and Rumsey (given in \cite{SV79}.) Once equicontinuity is established, the mean zero condition and the ergodic theorem are used to prove that $(g_n)_{n\geq1}$ converges uniformly to zero on bounded sets. This immediately implies the desired result. See Chapter 2 of \cite{jeffrey} for the complete proof.
\end{proof}

\section*{Appendix C}

\begin{proposition}
Suppose that $d=1$. For every $r\in\mathbb{R}$, $\mathbb{P}\left(\omega: E_o^\omega[\mathrm{e}^{r\tau_1}, \tau_1<\infty]<\infty\right)=1$ if and only if $\mathbb{P}\left(\omega: E_o^\omega[\mathrm{e}^{r\bar{\tau}_{-1}}, \bar{\tau}_{-1}<\infty]<\infty\right)=1$.
\end{proposition}

\begin{proof}
If $\mathbb{P}\left(\omega: E_o^\omega[\mathrm{e}^{r\tau_1}, \tau_1<\infty]<\infty\right)=1$, then there exists a $u_r:\Omega\times\mathbb{Z}\to\mathbb{R}^+$ that satisfies $u_r(\omega,0)=1$ and $$u_r(\omega,x)=\sum_{z\in\mathcal{R}}\pi(x,x+z)\mathrm{e}^ru_r(\omega,x+z)$$ for every $x\in\mathbb{Z}$ and $\mathbb{P}$-a.e.\ $\omega$. Therefore, $u_r(\omega,X_n)\mathrm{e}^{rn}$ is a martingale under $P_o^\omega$. By the stopping time theorem,
\begin{align*}
1=u_r(\omega,0)&=E_o^\omega\left[u_r\left(\omega,X_{\bar{\tau}_{-1}\wedge\tau_x}\right)\mathrm{e}^{r(\bar{\tau}_{-1}\wedge\tau_x)}\right]\\
&\geq E_o^\omega\left[u_r\left(\omega,X_{\bar{\tau}_{-1}}\right)\mathrm{e}^{r\bar{\tau}_{-1}},\bar{\tau}_{-1}<\tau_x\right]\\
&\geq\inf_{-B\leq z<0}u_r(\omega,z)E_o^\omega\left[\mathrm{e}^{r\bar{\tau}_{-1}},\bar{\tau}_{-1}<\tau_x\right]
\end{align*} for every $x\geq1$. Taking $x\to\infty$ shows that $E_o^\omega[\mathrm{e}^{r\bar{\tau}_{-1}}, \bar{\tau}_{-1}<\infty]<\infty$. The proof of the other direction is similar.
\end{proof}

\section*{Acknowledgements}
This work is part of my Ph.D. thesis. I am grateful to my advisor S.\ R.\ S.\ Varadhan for suggesting this topic and for many valuable discussions and ideas. I also thank F.\ Rassoul-Agha for a careful reading of an earlier version of the manuscript and for detailed and constructive comments.


\begin{thebibliography}{9}

\bibitem {alili}
\textsc{Alili, S.} (1999).
Asymptotic behaviour for random walks in random environments. \textit{J. Appl. Probab.}
\textbf{36} 334--349.

\bibitem {Bremont}
\textsc{Br\'emont, J.} (2007).
One-dimensional finite range random walk in random medium and invariant measure equation. \textit{ Ann. Inst. H. Poincar\'e Probab. Statist.} To appear.

\bibitem {CGZ}
\textsc{Comets, F.}, \textsc{Gantert, N.} and \textsc{Zeitouni, O.} (2000).
Quenched, annealed and functional large deviations for one dimensional random walk in random environment. \textit{Probab. Theory Related Fields}
\textbf{118} 65--114.

\bibitem {DeMasi}
\textsc{De Masi, A.}, \textsc{Ferrari, P. A.}, \textsc{Goldstein, S.} and \textsc{Wick, W. D.} (1989).
An invariance principle for reversible Markov processes with applications to random motions in random environments. \textit{J. Stat. Phys.}
\textbf{55} 787--855.

\bibitem {DemboZeitouni}
\textsc{Dembo, A.} and \textsc{Zeitouni, O.} (1998).
\textit{Large deviation techniques and applications}, 2nd ed.
Springer, New York.

\bibitem {KyFan}
\textsc{Fan, K.} (1953).
Minimax theorems. \textit{Proc. Natl. Acad. Sci. USA}
\textbf{39} 42--47.

\bibitem {GdH}
\textsc{Greven, A.} and \textsc{den Hollander, F.} (1994).
Large deviations for a random walk in random environment. \textit{Ann. Probab.}
\textbf{22} 1381--1428.

\bibitem {KV}
\textsc{Kipnis, C.} and \textsc{Varadhan, S.\ R.\ S.} (1986).
A central limit theorem for additive functionals of reversible Markov processes and applications to simple exclusion. \textit{Comm. Math. Phys.}
\textbf{104} 1--19.

\bibitem {KRV}
\textsc{Kosygina, E.}, \textsc{Rezakhanlou, F.} and \textsc{Varadhan, S. R. S.} (2006).
Stochastic homogenization of Hamilton-Jacobi-Bellman equations. \textit{Comm. Pure Appl. Math.}
\textbf{59} 1489--1521.

\bibitem {Kozlov}
\textsc{Kozlov, S. M.} (1985).
The averaging method and walks in inhomogeneous environments. \textit{Russian Math. Surveys (Uspekhi Mat. Nauk)}
\textbf{40} 73--145.

\bibitem {Olla}
\textsc{Olla, S.} (1994).
\textit{Homogenization of diffusion processes in random fields}.
Ecole Polytechnique, Palaiseau.

\bibitem {PV}
\textsc{Papanicolaou, G.} and \textsc{Varadhan, S.\ R.\ S.} (1981).
\textit{Boundary value problems with rapidly oscillating random coefficients} in "Random Fields", J. Fritz, D. Szasz editors, Janyos Bolyai series.
North-Holland, Amsterdam.

\bibitem {Firas}
\textsc{Rassoul-Agha, F.} (2003).
The point of view of the particle on the law of large numbers for random walks in a mixing random environment. \textit{Ann. Probab.}
\textbf{31} 1441--1463.

\bibitem {Rockafellar}
\textsc{Rockafellar, T.} (1972).
\textit{Convex analysis}, 2nd ed.
Princeton University, New Jersey.

\bibitem {roiter}
\textsc{Roitershtein, A.} (2007).
Transient random walks on a strip in a random environment. \textit{Ann. Probab.} To appear.

\bibitem {jeffrey}
\textsc{Rosenbluth, J.} (2006).
Quenched large deviations for multidimensional random walk in random environment: a variational formula. Ph.D. thesis, New York University.
arXiv:0804.1444v1

\bibitem {Rudin}
\textsc{Rudin, W.} (1991).
\textit{Functional analysis}, 2nd ed.
McGraw-Hill, New York.

\bibitem {Solomon}
\textsc{Solomon, F.} (1975).
Random walks in a random environment. \textit{Ann. Probab.}
\textbf{3} 1--31.

\bibitem{SV79}
\textsc{Stroock, D.\ W.} and \textsc{Varadhan, S.\ R.\ S.} (1979).
\textit{Multidimensional Diffusion Processes}.
Springer-Verlag, Berlin.

\bibitem {Sznitman}
\textsc{Sznitman, A. S.} (2002).
\textit{Lectures on random motions in random media} in "Ten Lectures on Random Media", DMV-Lectures \textbf{32}.
Birkh\"{a}user, Basel.

\bibitem {Raghu}
\textsc{Varadhan, S.\ R.\ S.} (2003).
Large deviations for random walks in a random environment. \textit{Comm. Pure Appl. Math.}
\textbf{56} 1222--1245.

\bibitem {rwrereview}
\textsc{Zeitouni, O.} (2006).
Random walks in random environments. \textit{J. Phys. A: Math. Gen.}
\textbf{39} R433--464.

\bibitem {Zerner}
\textsc{Zerner, M. P. W.} (1998).
Lyapounov exponents and quenched large deviations for multidimensional random walk in random environment. \textit{Ann. Probab.}
\textbf{26} 1446--76.

\end{thebibliography}
\end{document}